\documentclass[11pt,a4paper]{article}

\usepackage{packages}

\numberwithin{equation}{section}
\numberwithin{table}{section}
\numberwithin{figure}{section}

\newtheorem{theorem}{Theorem}[section]

\newtheorem{definition}[theorem]{Definition}
 
\newtheorem{prop}[theorem]{Proposition}
\newtheorem{lemma}[theorem]{Lemma}
\crefname{lemma}{lemma}{lemmata}
\Crefname{lemma}{Lemma}{Lemmata}

\crefname{corollary}{corollary}{corollaries}
\Crefname{corollary}{Corollary}{Corollaries}

\theoremstyle{definition}
\newtheorem{remark}[theorem]{Remark}
\AtEndEnvironment{remark}{\hfill\ensuremath{\diamondsuit}}
\theoremstyle{definition}
\newtheorem{example}[theorem]{Example}
\AtEndEnvironment{example}{\hfill\ensuremath{\bigcirc}}

\usepackage{glossary}
\usepackage{mymacros}

\title{The Pantelides algorithm for delay differential-algebraic equations}
\author{Ines Ahrens\footnotemark[1]~\footnotemark[2] \and Benjamin Unger\footnotemark[1]~\footnotemark[3]}
\date{\today}

\begin{document}

\maketitle
\renewcommand{\thefootnote}{\fnsymbol{footnote}}
\footnotetext[1]{Institut f\"ur Mathematik,
Technische Universität Berlin, Str.\ des 17.~Juni~136,
10623~Berlin,
Federal Republic of Germany,
\texttt{\{ahrens,unger\}@math.tu-berlin.de}. }
\footnotetext[2]{
The work of this author is supported by the Deutsche Forschungsgemeinschaft (DFG, 
German Research Foundation) - 384950143/ GRK2433}
\footnotetext[3]{
The work of this author is supported by the DFG Collaborative Research Center 910 \emph{Control of self-organizing nonlinear systems: Theoretical methods and concepts of application}, project number 163436311.}
%
\begin{abstract}
	We present a graph-theoretical approach that can detect which equations of a delay differential-algebraic equation (DDAE) need to be differentiated or shifted to construct a solution of the DDAE. Our approach exploits the observation that differentiation and shifting are very similar from a structural point of view, which allows us to generalize the Pantelides algorithm for differential-algebraic equations to the DDAE setting. The primary tool for the extension is the introduction of equivalence classes in the graph of the DDAE, which also allows us to derive a necessary and sufficient criterion for the termination of the new algorithm. 
\vskip .3truecm

\vskip .3truecm

{\bf AMS(MOS) subject classification:} 34A09, 34K32, 65L80
\end{abstract}

\section{Introduction}
\renewcommand*{\thefootnote}{\arabic{footnote}}
\setcounter{footnote}{0}

We study \glspl{DDAE} of the form
\begin{equation}
	\label{eq:nlDDAE}
	F(t,\state(t),\dot{\state}(t),\state(t-\tau)) = 0
\end{equation}
on the time interval $\timeInt$ with $\tau>0$. The \gls{DDAE} \eqref{eq:nlDDAE} is equipped with an (functional) initial condition of the form
\begin{equation*}
	\state(t) = \hist(t)\qquad\text{for}\ t\in[-\tau,0].
\end{equation*}
Differential equations with delay arise in various applications, where the current rate of change does not only depend on the current state but also on past information, for instance, population dynamics, infection disease models and laser dynamics. We refer to \cite{Ern09} and the references therein. Besides, if a dynamical system is controlled based on the current state of the system -- a so-called feedback control law -- then such a controller often requires some time to measure the current state and to compute the control action, which again results in a delay equation. On the other hand, modern modeling packages such as \textsc{modelica}\footnote{\url{https://www.modelica.org/}} or \textsc{Matlab/Simulink}\footnote{\url{https://www.mathworks.com/}} automatically generate dynamical systems with constraints that arise due to interface conditions or conservation laws. As a consequence, we have to deal with equations of the form \eqref{eq:nlDDAE}, which combine the features of \glspl{DDE} \cite{HalL93,BelZ03,Glu02,BelC63} and \glspl{DAE} \cite{KunM06,Pet82}. Further applications of \glspl{DDAE} are the analysis of hybrid numerical-experimental models arising in real-time dynamic substructuring \cite{Ung20}, blood flow models \cite{BorKT19}, analysis \cite{AltMU19} and construction \cite{BelGZ99,BelGZ00} of numerical time integration schemes, and the realization of transport phenomenon \cite{SchU16, SchUBG18, FosSU19}. Let us emphasize that our framework is not restricted to a single time delay, since multiple commensurate delays can be rewritten in the form \eqref{eq:nlDDAE} by introducing new variables \cite{Ha15}.

One of the main difficulties for \glspl{DDAE} is the fact that solutions may depend on derivatives of the function $F$ in \eqref{eq:nlDDAE} as well as on evaluations of $F$ at future time points \cite{HaM12,Cam95}, i.e., one has to study the interplay of the differentiation operator $\mathrm{d}/\mathrm{d}t$ and the (time) shift operator $\shift{\!}{\tau}$ \cite{HAMS14}, defined via $(\shift{x}{\tau})(t) = x(t+\tau)$. In particular, it is important to understand which equations in \eqref{eq:nlDDAE} need to be differentiated and which need to be shifted. For linear systems, this is investigated for instance in \cite{Cam95,HaM12,HAMS14,HaM16,Ung18,TreU19}. Note that already in the linear case, existence and uniqueness results for \glspl{DDAE} require a distributional solution concept \cite{TreU19}. To obtain continuous solutions one either has to impose restrictions on the \glspl{DDAE} \cite{HaM12,HaM16,Ung18} or on the history function \cite{Ung18,Ha18}. For nonlinear \glspl{DDAE}, solutions are establish for certain classed in \cite{AscP95,Ung20}. 

The main idea to establish solutions of \eqref{eq:nlDDAE}, as described in the literature above, is via integration on successive time intervals $[i\tau,(i+1)\tau)$, the so-called \emph{method of steps} \cite{BelZ03,HaM16,Cam80}. At each interval one thus has to solve the \gls{DAE} that is obtained by substituting the delayed variable in \eqref{eq:nlDDAE} with the already computed solution. A necessary condition for this procedure to succeed is that the \gls{DAE} is regular. Unfortunately, this is not necessarily satisfied even if the initial value problem for \eqref{eq:nlDDAE} is uniquely solvable \cite{HaM12}. Instead, we seek a reformulation of \eqref{eq:nlDDAE} such that the \gls{DAE} that has to be solved in each interval is regular and has small index. In the literature, such a reformulation is constructed either with a compress-and-shift algorithm \cite{Cam95,TreU19} or via a combined shift and derivative array \cite{HaM16}. In both cases one has to shift certain equations and differentiate certain equations.

In this paper, we propose a graph-theoretical approach to determine which equations need to be differentiated and which equations need to be shifted. To this end, we revisit the \defn{Pantelides} algorithm \cite{Pan88} for \glspl{DAE} in \cref{sec:prelim}, a methodology that uses the information which variable appears in which equation to determine which equations need to be differentiated. Our main contributions are the following:
\begin{enumerate}
	\item We introduce equivalence classes (cf.\ \Cref{def:graphDAE}) in the graph of an equation, which allows us to derive a criterion -- see \Cref{thm:TerminationPDAE} -- when the Pantelides algorithm terminates. We show that this criterion is equivalent to the criterion presented in \cite{Pan88} and thus provides another interpretation, that enables us to extend the algorithm to the \gls{DDAE} case.
	 \item In \cref{subsec:diffvsshift} we illustrate that, from a structural point of view, the operations differentiation and shifting are similar, such that we can essentially use the same procedure to determine the equations that need to be shifted or differentiated, respectively. 
    \item We introduce a shifting and differentiation graph in \Cref{def:shiftDiffGraph} and show that differentiation does not affect the shifting graph (\Cref{prop:PDDAEdiffDoesntEffectShift}). As a consequence, we follow the strategy in \cite{Cam95,TreU19}, and first determine the equations that need to be shifted. It turns out that during this procedure some of the equations need to be differentiated and we propose linear integer programs to determine which equations need to be differentiated how many times. \Cref{thm:linearIntegerProgram} details that each integer program is equivalent to a standard linear program that can be solved with standard methods. 
	\item Our methodology is summarized in \Cref{alg:PantDDAE}. We conclude this manuscript with a detailed analysis of \Cref{alg:PantDDAE} in \cref{sec:termination}: We show in \Cref{thm:TerminationPantDDAE3} that \Cref{alg:PantDDAE} terminates if and only if the \gls{DDAE} is structurally nonsingular with respect to a certain equivalence relation. Moreover \Cref{thm:noShift} details in which cases \Cref{alg:PantDDAE} concludes that no equations need to be shifted and that in this situation it produces the same result as if we apply the original Pantelides algorithm to the \gls{DDAE} \eqref{eq:nlDDAE} where we replace $\shift{x}{-\tau}$ with a function parameter $\lambda$. 
\end{enumerate}

We are convinced that our methodology can be combined with a dummy derivative approach \cite{MatS93} or with algebraic regularization techniques as proposed in \cite{SchS16a,SchS16b}, and thus serves as the first step to a numerical method for general nonlinear \glspl{DDAE}. We emphasize that -- similar to the original Pantelides algorithm \cite{Pan88} -- our method not always determines the correct number of differentiations and shifts. For \glspl{DAE} this fact is accounted for with a success check \cite{Pry01,PryNeTa15} and an algebraic analysis \cite{SchS16a}. A similar validation of the results for \glspl{DDAE} is ongoing research.

\paragraph*{Notation}
The natural numbers, the nonnegative integers, and the reals are denoted by $\mathbb{N}$, $\mathbb{N}_0$, and $\mathbb{R}$, respectively. For a set $X$ the power set of $X$ is denoted by $\mathcal{P}(X)$. For a differentiable function $f\colon\I\to\mathbb{R}^n$ we use the notation $\dot{f} \vcentcolon= \tfrac{\mathrm{d}}{\mathrm{d}t} f$ to denote the derivative from the right\footnote{For delay equations it is standard to use the derivative from the right \cite{Cam95,HalL93} since the history function may not be linked smoothly to the solution of the associated initial value problem.} with respect to the (time) variable $t$. Similarly, the shift operator $\shift{\!}{\tau}$ is defined as $(\shift{f}{\tau})(t) = f(t+\tau)$. As a consequence, the \gls{DDAE} \eqref{eq:nlDDAE} can be conveniently written as
\begin{displaymath}
	F(t,x,\dot{x},\shift{x}{-\tau}) = 0.
\end{displaymath}
The partial derivative of $F$ with respect to $x$, $\dot{x}$, and $\shift{x}{-\tau}$ is denoted by $\tfrac{\partial F}{\partial x}$, $\tfrac{\partial F}{\partial \dot{x}}$, and $\frac{\partial F}{\partial \shift{x}{-\tau}}$, respectively.  The union of two sets $A$ and $B$ is denoted by $A\cup B$. If, moreover, the intersection $A\cap B$ of $A$ and $B$ is empty, then we write $A\dot{\cup}B$. The number of elements that are contained in the set $A$ is denoted by $|A|$. For an equivalence relation $R \subseteq A\times A$ we denote the set of equivalence classes by $\eqclass{A}{R}$. For a subset $\widehat{A}\subseteq A$ we define $\eqclass{\widehat{A}}{R}\vcentcolon= \eqclass{\widehat{A}}{\widehat{R}}$ with
\begin{displaymath}
	\widehat{R} \vcentcolon= \{(x,y)\in R \mid x,y\in \widehat{A}\}.
\end{displaymath}

\section{Preliminary Results}
\label{sec:prelim}
A standard approach to solve differential equations with delay is the method of steps \cite{BelZ03}. In the context of \glspl{DDAE}, this requires to solve a sequence of initial value problems with a \gls{DAE}, see for instance \cite{Ung18}. More precisely, we study the \gls{DAE}
\begin{subequations}
	\label{eq:IVP}
\begin{align}
	\label{eq:nlDAE}
	F(t, x, \dot{x}) = 0,
\end{align}
where $F\colon\I \times \D_x \times \D_{\dot{x}} \to \R^n$, $\I \subseteq \R$ is an interval and $\D_x, \D_{\dot{x}} \subseteq \Rn$ are open, together with the initial condition
\begin{align}
	\label{eq:initialCondition}
	x(t_0) = x_0.
\end{align}
\end{subequations}
One of the main differences of \glspl{DAE} compared to \glspl{ODE} is that the partial derivative of $F$ with respect to $\dot{x}$, denoted via $\tfrac{\partial F}{\partial \dot{x}}$, may be singular. It is well-known that in general solutions of \eqref{eq:IVP} depend on derivatives of \eqref{eq:nlDAE}, see for instance \cite{BreCP96,Pet82,KunM06}. Following \cite{Cam87a}, we introduce the so-called \defn{derivative array} of level~$\ell$
\begin{displaymath}
	\mathcal{D}_\ell(t,x,\dot{x},\ldots,x^{(\ell+1)}) \defEqual \begin{bmatrix}
		F(t,x,\dot{x})\\
		\frac{\mathrm{d}}{\mathrm{d}t} F(t,x,\dot{x})\\
		\vdots\\
		\frac{\mathrm{d}^\ell}{\mathrm{d}t^\ell} F(t,x,\dot{x})
	\end{bmatrix},
\end{displaymath}
which can be used to determine the solution of \eqref{eq:IVP} for large enough $\ell$. The \defn{strangeness index} \cite{KunM98,KunM06} for instance uses $\mathcal{D}_\ell$ to determine matrix functions, which can be used to construct a reformulated version of \eqref{eq:nlDAE} that is suitable for numerical time integration. Although this procedure is tailored to numerical methods, it suffers from two issues: First, in a large-scale setting, the computation of the matrix functions may become computationally expensive, since all equations in \eqref{eq:nlDAE} are differentiated instead of only the required equations. Second, the computation of the matrix functions requires several numerical challenging rank decisions. 

To overcome these issues, \cite{SchS16a,SchS16b} propose to combine such algebraic regularization techniques with so-called structural analysis tools, such as the $\Sigma$-method \cite{Pry01, PryNeTa15} and the Pantelides algorithm \cite{Pan88}. Hereby, structure is understood as the information which variable appears in which equations. In other words, the structure defines a bipartite graph $G = (\VE \dot{\cup} \VV,E)$ where the set of vertices $ \VE \dot{\cup} \VV$ is given by the set of equations $\VE$ and the set of variables $\VV$. An edge $\{F,x\}\in E$ implies that the variable $x\in \VV$ appears in the equation $F\in \VE$ (see \Cref{def:graphDAE}). Note that we need to distinguish between a variable and the name of the variable in the definition of the bipartite graph. We, therefore, introduce the language   
\begin{align*}
	L \defEqual \left\{  \shift{\xi_{j}^{(i)}}{k \tau}  \mid \xi \in \{F,x\},\,j \in \N,\, i \in \N_0,\, k \in \{-1\} \cup \N_0 \right\},
\end{align*}
which accounts for differentiation and also for shifting, which is required for the upcoming analysis of the \gls{DDAE} \eqref{eq:nlDDAE}.
For notational convenience we write $\xi_j$ instead of $\xi_j^{(0)}$,  $\dot{\xi}_j$ instead of $\xi_j^{(1)}$ and $\xi_j^{(i)} $ instead of $\shift{\xi_j^{(i)}}{0 \tau}$. To map a vector of variable or equation names (as defined in $L$) to a set containing these elements, we define 
\begin{align*}
\set \colon L^N \rightarrow \mathcal{P}(L), \qquad (a_1,\dots,a_N) \mapsto  \{a_1, \dots, a_N \}.	
\end{align*}
For the remainder of this paper we identify a variable by its name and do not distinguish between both. 

The main idea of the Pantelides algorithm \cite{Pan88} is to answer the question, which equation determines which variable. More precisely, we are interested in matching each equation with a \defn{highest derivative}, i.e., a variable $x_i^{(k)}$ such that $x_i^{(k)}$ occurs in this equation but $x_i^{(k+\ell)}$ for $\ell >0$ does not occur in any equation. For instance, in an \gls{ODE} of the form $\dot{x} = f(t,x)$ we can simply assign the $i$th equation to $\dot{x}_i$. The matching of an equation to a highest derivative is possible only if we can match each equation to a different variable regardless of the order of differentiation. Thus, it is desirable to group variables of different differentiation levels, which motivates the following definition.
\begin{definition}
	\label{def:graphDAE}
	Consider the equation
	\begin{align}
		\label{eq:generalEquation}
		F(t,\allvars) = 0
	\end{align}
	with $F\colon \mathbb{I}\times \mathbb{D}_\allvars \to \R^M$ and $\mathbb{D}_\allvars\subseteq\R^N$.
	\begin{enumerate}
		\item The set
			\begin{displaymath}
				\setvars \defEqual \left\{\elementsetvars\in\set(\allvars) \,\left|\,  \elementsetvars \text{ appears in } F \right.\right\}
			\end{displaymath}
			is called the \defn{set of variables} of the equation \eqref{eq:generalEquation}.
		\item Let $\setvars$ denote the set of variables of \eqref{eq:generalEquation} and $\widetilde{\setvars}\subseteq\setvars$. The \defn{graph} of \eqref{eq:generalEquation} over the equivalence relation $R \subseteq \widetilde{\setvars}\times\widetilde{\setvars}$ is defined to be the bipartite graph $G = (\VE \dot{\cup} \VV, E)$ where $\VE \defEqual \set(F )$ is the \defn{set of equations}, $ \VV \defEqual \eqclass{\widetilde{\setvars}}{R}$ denotes the \defn{set of variables with respect to $R$} and 
			\begin{align}
				\label{eq:edgesGraphEquation}
				E\defEqual  & \left\{\{V, w\} \,\left|\,  V \in \VV, w \in \VE :\text{there exists } v \in V \text{ which appears in } w \right.\right\}.
			\end{align}
		\item Let $\setvars$ denote the set of variables for the \gls{DAE} \eqref{eq:nlDAE} (with $\allvars^T = \begin{bmatrix} x^T & \dot{x}^T\end{bmatrix}$) and $R \defEqual \{(\elementsetvars,\elementsetvars) \mid \elementsetvars\in \setvars\} \subseteq \setvars \times \setvars$. The bipartite graph $G = (\VE\dot{\cup}\VV, E) $ is called \defn{graph of the \gls{DAE}} \eqref{eq:nlDAE} if $G$ is the graph of \eqref{eq:nlDAE} over $R$. 
	\end{enumerate}
\end{definition}

\begin{remark}
	Note that if a variable appears in an equation, that equation may not depend on that variable. 
	For example, in the expression $\sin^2 (x_1) + \cos^2 (x_1)$ the variable $x_1$ appears but the expression does not depend on $x_1$ \cite{TanNP17}. In general, it is not possible to determine if an expression truly depends on a variable \cite{Ri68}. 
\end{remark}

Each equivalence class in the set of variables with respect to $R$ for a graph of a \gls{DAE} contains one variable. Thus for \glspl{DAE} we can identify the set of variables and the set of variables with respect to $R$ with each other. In general, this is not possible. Nevertheless, we do not distinguish between a set of variables and a set of variables with respect to $R$ since it is implied by the context.
	
\begin{example}
	\label{ex:DAEgraphExample}
	In the \gls{DAE} 
	\begin{equation}
		\label{eq:Ex1}
		\begin{aligned}
			0 &= x_1 + f_1,\qquad &
			\dot{x}_1 & = x_2 + f_2,\qquad &
			\dot{x}_2 & = x_3 + f_3 
		\end{aligned}		
	\end{equation}
	we denote the $i$th equation by $F_i$ for $i=1,2,3$. 
	The graph $G = (\VE \dot{\cup} \VV, E)$ of the \gls{DAE} \eqref{eq:Ex1} given by 
	\begin{align*}
		\VE &\defEqual \left\{ F_1, F_2, F_3 \right\}\subseteq \mathcal{P}(L),\\
		\VV &\defEqual \left\{ \{x_1\}, \{x_2\}, \{x_3\}, \{\dot{x}_1\}, \{\dot{x}_3\} \right\}\subseteq\mathcal{P}(L),\\
		E &\defEqual \left\{ \left\{F_1,\{x_1\} \right\},\left\{F_2,\{x_2\} \right\}, \left\{F_2,\{\dot{x}_1\}\right\}, \left\{F_3,\{x_3\} \right\}, \left\{F_3,\{\dot{x}_2\} \right\}  \right\}
	\end{align*}		
	is visualized in \Cref{fig:DAEEx1G1} (where we represent the equivalence classes by listing all its elements). By standard abuse of notation we do not distinguish between a graph and its visualization.
\end{example}

\begin{figure}[t]
	\begin{subfigure}[t]{0.45\textwidth}
		\centering
		\begin{tikzpicture} 
		\enode{1}{$F_1$}
		\enode{2}{$F_2$}
		\enode{3}{$F_3$}
		\vnode{1}{$x_1$}{}
		\vnode{2}{$x_2$}{}
		\vnode{3}{$x_3$}{}
		\vnode{4}{$\dot{x}_1$}{}
		\vnode{5}{$\dot{x}_2$}{}
		\edge{e1}{v1}
		\edge{e2}{v2}
		\edge{e2}{v4}		
		\edge{e3}{v5}	
		\edge{e3}{v3}	
		\end{tikzpicture}	
		\caption{Graph of the \gls{DAE}}
		\label{fig:DAEEx1G1}
	\end{subfigure}
	\qquad
	\begin{subfigure}[t]{0.45\textwidth}
		\centering
		\begin{tikzpicture} 
		\enode[try]{1}{$F_1$}
		\enode{2}{$F_2$}
		\enode{3}{$F_3$}
		\vnode[low]{1}{$x_1$}{}
		\vnode[low]{2}{$x_2$}{}
		\vnode{3}{$x_3$}{}
		\vnode{4}{$\dot{x}_1$}{}
		\vnode{5}{$\dot{x}_2$}{}
		\edge[low]{e1}{v1}
		\edge[low]{e2}{v2}
		\edge[assign]{e2}{v4}		
		\edge[assign]{e3}{v5}	
		\edge{e3}{v3}	
		\end{tikzpicture}	
		\caption{Assignment after deleting $x_1$ and $x_2$}
		\label{fig:DAEEx1G1-assign}
	\end{subfigure}\\
	\begin{subfigure}[t]{0.45\textwidth}
		\centering
		\begin{tikzpicture} 
		\enode[try]{1}{$\dot{F}_1$}
		\enode[try]{2}{$F_2$}
		\enode{3}{$F_3$}
		\vnode[low]{1}{$x_1$}{}
		\vnode[low]{2}{$x_2$}{}
		\vnode{3}{$x_3$}{}
		\vnode{4}{$\dot{x}_1$}{}
		\vnode{5}{$\dot{x}_2$}{}
		\edge[low]{e1}{v1}
		\edge[assign]{e1}{v4}
		\edge[low]{e2}{v2}
		\edge{e2}{v4}		
		\edge[assign]{e3}{v5}	
		\edge{e3}{v3}	
		\end{tikzpicture}	
		\caption{Graph for \eqref{eq:Ex1} after differentiating $F_1$}
		\label{fig:DAEEx1G2}	
	\end{subfigure}
	\qquad
	\begin{subfigure}[t]{0.45\textwidth}
	\centering
	\begin{tikzpicture} 
	\enode{1}{$\ddot{F}_1$}
	\enode{2}{$\dot{F}_2$}
	\enode{3}{$F_3$}
	\vnode[low]{1}{$x_1$}{}
	\vnode[low]{2}{$x_2$}{}
	\vnode{3}{$x_3$}{}
	\vnode{4}{$\dot{x}_1$}{}
	\vnode{5}{$\dot{x}_2$}{}
	\vnode{6}{$\ddot{x}_1$}{}
	\edge[low]{e1}{v1}
	\edge[low]{e1}{v4}
	\edge[assign]{e1}{v6}
	\edge[low]{e2}{v2}
	\edge[low]{e2}{v4}
	\edge[assign]{e2}{v5}
	\edge{e2}{v6}			
	\edge{e3}{v5}	
	\edge[assign]{e3}{v3}		
	\end{tikzpicture}	
	\caption{Graph for \eqref{eq:Ex1} after differentiating $F_1$ twice and $F_2$ once}
	\label{fig:DAEEx1G3}	
\end{subfigure}
	\caption{Visualization of \Cref{ex:DAEgraphExample}}
\end{figure}

In the language of graph theory, the assignment of an equation to a variable can be expressed as a \defn{matching}  $\mathcal{M} \subseteq E$ (in the literature also called \defn{assignment}) in a graph $(V,E)$, which is a subset of edges, such that for all vertices $v \in V$ with  $v \in e_1 \in \mathcal{M}$ and $v \in e_2 \in \mathcal{M}$ we have $e_1 = e_2$, i.e., no vertex appears in more than one edge. An edge $e = \{v_1, v_2\} \in \mathcal{M}$ is called a \defn{matching edge} and we say that $v_1$ and $v_2$ are \defn{matched}. A vertex which does not occur in any matching edge is called an \defn{exposed vertex} with respect to the matching. In a bipartite graph $G = (\VE \dot{\cup} \VV, E)$ a \defn{maximal matching} is a matching such that no $F \in \VE$ is exposed.
	
\begin{example}
	\label{ex:DAEgraphExample-step2}
	In the graph in \Cref{ex:DAEgraphExample} the variables $x_1, x_2$ are not highest derivatives such that we can (temporarily) delete these variables from the graph, which we visualize by using gray edges in \Cref{fig:DAEEx1G1-assign}. A matching is for instance given by $\mathcal{M} = \{\{F_2,\{\dot{x}_1\}\},\{F_3,\{\dot{x}_2\}\}\}$, which we depict in \Cref{fig:DAEEx1G1-assign}	via thick blue line.
	We immediately notice that $F_1$ is an exposed vertex (independent of any possible matching and indicated by a black vertex) and thus no maximal matching is possible.
\end{example}

\begin{remark}
	\label{rem:relationHighestDerivative}
	Let $\setvars$ denote the set of variables for the \gls{DAE} \eqref{eq:nlDAE} (cf.\ \Cref{def:graphDAE}). Define
	\begin{align}
		\label{eq:highestDerivative}
		\widetilde{\setvars} \defEqual \{\elementsetvars\in \setvars \mid \elementsetvars \text{\ is highest derivative}\}
	\end{align}
	and $\widetilde{R}_{\mathrm{equal}} \defEqual \{(\elementsetvars,\elementsetvars) \mid \elementsetvars\in\widetilde{\setvars}\}$. Then the graph of the \gls{DAE} that is obtained after deleting all variables that are not highest derivative is the graph of \eqref{eq:nlDAE} over $\widetilde{R}_{\mathrm{equal}}$.
\end{remark}

Suppose that, after deleting variables that are no highest derivatives, we cannot find a maximal matching and thus have an exposed vertex $F_i$. Then one of the following applies: The function $F_i$ does not depend on any variable, implying that in general, the DAE is not regular (see \cite{KunM98} for further details). Otherwise, the function $F_i$ determines a variable that is not a highest derivative, and thus $F_i$ needs to be differentiated to construct a solution. More precisely, we replace $F_i$ by $\frac{\mathrm{d}}{\mathrm{d}t} F_i(t,\allvars) = 0$ and add possible new variables. Note that all variables that appear in $F_i$ appear with one additional derivative in $\frac{\mathrm{d}}{\mathrm{d}t} F_i(t,\allvars) = 0$ and thus the original variables cannot be highest derivatives and can thus be ignored in the following.

\begin{example}
	\label{ex:DAEgraphExample-step3}
	Continuing with \Cref{ex:DAEgraphExample-step2}, we replace $F_1$ with the differentiated equation $\frac{\mathrm{d}}{\mathrm{d}t} F_1(t,\allvars) = 0$, which we denote by $\dot{F}_1$. Note that there is a (gray) edge in \Cref{fig:DAEEx1G2} between $\dot{F}_1$ and $x_1$, since in our structural approach we do not work with the actual equation but only with the information that  $F_1$ depends on $x_1$. The chain rule then implies that $\dot{F}_1$ depends on $x_1$ and $\dot{x}_1$. A possible matching is given by $\mathcal{M} = \{\{\dot{F}_1,\{\dot{x}_1\}\},\{F_3,\{\dot{x}_2\}\}\}$, such that $F_2$ is an exposed vertex with respect to $\mathcal{M}$. The resulting graph is depicted in \Cref{fig:DAEEx1G2}. Continuing with our strategy for \Cref{ex:DAEgraphExample-step3} implies that we need to differentiate equation $F_2$. This time, the situation is slightly different than before, since there exists a coupling between $\dot{F}_1$ and $F_2$ via the variable $\dot{x}_1$ indicated by a black vertex, i.e., we have a path from $\dot{F}_1$ to $F_2$ via $\dot{x}_1$.
\end{example}

A \defn{path} is a sequence of edges $(e_1, \dots, e_n)$ in a graph $G= (V,E)$ such that $e_i = \{v_{i-1} ,v_i\} \in E$, where $v_0, \dots, v_n \in V$ are distinct. An \defn{alternating path} with respect to a matching is a path with alternating non-matching and matching edges that starts with a non-matching edge. An \defn{augmenting path} with respect to a matching is an alternating path that ends with a non-matching edge. It is easy to see that whenever an augmenting path exists, we can find another matching that includes more equations than the previous matching and hence we cannot determine (at this point) that the equations need to be differentiated. Hence we only have to differentiate whenever we cannot find an augmenting path. In this case, we have to differentiate all equations that are connected with the exposed equation via an alternating path. For further details, we refer to \cite{Pan88}.

\begin{example}
	In \Cref{ex:DAEgraphExample-step3} there exists no augmenting path that starts in $F_2$ (cf. \Cref{fig:DAEEx1G2}). Thus, we have to differentiate $F_2$ and all equations $F_j$ with an alternating path from $F_2$ to $F_j$. In \Cref{fig:DAEEx1G2} these equations are denoted by a black dot. In this case the path $(F_2,\{\dot{x}_1\},\dot{F}_1)$ implies that we have to differentiate $\dot{F}_1$ and $F_2$. Note that in the resulting graph, given in \Cref{fig:DAEEx1G3}, we find the maximal matching
	\begin{displaymath}
		\mathcal{M} = \{\{\ddot{F}_1,\{\ddot{x}_1\}\}, \{\dot{F}_2,\{\dot{x}_2\}\},\{F_3,\{x_3\}\}\}
	\end{displaymath}
	and can thus stop.
\end{example}

Before we present an algorithmic summary of the outline methodology, let us recall the main idea of the Pantelides algorithm, namely to match each equation in \eqref{eq:nlDAE} with a highest derivative.
In more detail, we have considered the highest derivative of each variable (for the corresponding equivalence relation see \Cref{rem:relationHighestDerivative}) to determine which equations need to be differentiated. Let $\{k_1,\ldots,k_m\} \subseteq\{1,\ldots,M\}$ denote the equations that need to be differentiated. Then we define a \defn{subset of equations} of \eqref{eq:generalEquation} via
\begin{align}
	\label{eq:subsetEq}
	\widehat{F}(t, \widehat{\allvars}) \defEqual \begin{bmatrix}
		F_{k_1}(t,\allvars)\\
		\vdots\\
		F_{k_M}(t,\allvars)
	\end{bmatrix} = 0,
\end{align}
with the understanding that $\widehat{\allvars}$ contains only the subset of variables of $\allvars$ that appears in any equation of $\widehat{F}$. It is clear that we cannot match each equation in \eqref{eq:subsetEq} with a highest derivative if there are more equations than highest derivative variables, i.e., equivalence classes in $\eqclass{\left( \set (\hat{\allvars}) \cap \tilde{\setvars} \right) }{R_{\mathrm{equal}}}$ with $\tilde{\setvars}$ defined as in \eqref{eq:highestDerivative}.

\begin{definition}
	\label{def:structuralSingular}
	A \defn{subset} \eqref{eq:subsetEq} of \eqref{eq:generalEquation} is called \defn{structurally  singular} with respect to an equivalence class $\eqclass{\widehat{\setvars}}{\widehat{R}}$ with $\widehat{\setvars} \subseteq \set(\widehat{\allvars}) $ and relation $\widehat{R}$ if
	\begin{align*}
	|\set(\widehat{F}) | > |\eqclass{\widehat{\setvars}}{\widehat{R}} |. 
	\end{align*} 
	The \defn{equation} \eqref{eq:generalEquation} is called \defn{structurally singular} with respect to an equivalence class $\eqclass{\setvars}{R}$ with $\setvars \subseteq \set(\allvars)$ if it contains a structurally singular subset \eqref{eq:subsetEq} with respect to $\eqclass{(\setvars \cap \set(\widehat{\allvars}))}{R}$\footnote{Recall that we use $\eqclass{(\setvars \cap \set(\widehat{\allvars}))}{R}$ to denote the set of equivalence classes $\eqclass{(\setvars \cap \set(\widehat{\allvars}))}{\widehat{R}}$ with the restricted equivalence relation $\widehat{R} \vcentcolon= \{(x,y)\in R \mid x,y\in \setvars\cap \set(\widehat{\allvars})\}$ }.
	A structurally singular subset \eqref{eq:subsetEq} is called \gls{MSS} with respect to an equivalence class $\eqclass{\setvars}{R}$ if none of its proper subsets \eqref{eq:subsetEq} is structurally singular with respect to $\eqclass{(\setvars \cap \set(\widehat{\allvars}))}{R}$.
\end{definition}

\begin{remark}
 	The notion of structural singularity of a subset of equations is defined in \cite{Pan88}. In contrast to \cite{Pan88} we restrict ourselves to the set of variables that are actually contained within this subset of equations.
\end{remark}

\begin{prop}
	\label{lem:AugmentpathMSS}
	Consider the equation \eqref{eq:generalEquation}, let $\setvars$ denote the set of variables of \eqref{eq:generalEquation}. For a subset $\tilde{\setvars}\subseteq\setvars$ with equivalence relation $R\subseteq\tilde{\setvars}\times\tilde{\setvars}$ consider the graph $G$ of \eqref{eq:generalEquation} over $\eqclass{\tilde{\setvars}}{R}$ with a matching $\mathcal{M}$. Let $F_j$ be an exposed node. Then one of the following is true:
	\begin{enumerate}
		\item There exists an augmenting path starting in $F_j$.
		\item The set of equations $\hF(t, \widehat{\allvars}) = 0 $ associated with 
			\begin{equation}
				\label{eq:MSS}
				C_{F_j} \defEqual \left\{F_k \in \VE \mid \text{there exists an alternating path between } F_j \text { and } F_k \text{ in } G \right\}
			\end{equation}
			is \gls{MSS}  with respect to $\eqclass{(\tilde{\setvars} \cap \set(\hat{\allvars}) )}{R}$.
	\end{enumerate}
\end{prop}

\begin{proof}
	Since all equations in $C_{F_j}$ have an alternating path to $F_j$ we conclude that $F_j$ is the only exposed equation in $C_{F_j}$. This implies 
	\begin{equation}
		\label{eq:proofAugmentingPathMSS}
		|C_{F_j}|-1 \leq |\eqclass{(\tilde{\setvars}\cap\set(\widehat{\allvars}))}{R}|
	\end{equation}		
	Suppose that there is no augmenting path that starts in $F_j$. Assume that the inequality in \eqref{eq:proofAugmentingPathMSS} is strict. Then  there exists at least one equivalence class in $\eqclass{(\tilde{\setvars}\cap\set(\widehat{\allvars}))}{R}$ that is exposed. By definition of $\eqclass{(\tilde{\setvars}\cap\set(\widehat{\allvars}))}{R}$ and $C_{F_j}$ this implies that there exists an augmenting path that starts in $F_j$, a contradiction. We conclude that we have equality in \eqref{eq:proofAugmentingPathMSS} and thus $C_{F_j}$ is structurally singular.
	
	Consider a proper subset $\tilde{C}_{F_j} \subsetneq C_{F_j}$. 
	If $F_j \not\in \tilde{C}_{F_j}$ then all $F_k \in \tilde{C}_{F_j}$ are contained in a matching edge. Thus the set is structurally nonsingular. If $F_j \in \tilde{C}_{F_j}$ there exists an edge between a variable $x_k \in \eqclass{(\tilde{\setvars} \cap \set(\widehat{\allvars}))}{R}$ and an equation $F_k \in \tilde{C}_{F_j}$ which is not a matching edge. Thus $\tilde{C}_{F_j}$ is structurally nonsingular. This completes the proof.
\end{proof}

\Cref{lem:AugmentpathMSS} allows us to implement the outlined strategy as follows. Suppose that we have already found a matching for the first $j-1$ equations. For the $j$th equation, we try to find an augmenting path. If there exists an augmenting path, then we can directly generate a matching that includes the first $j$ equations. If not, then we differentiate all equations within the set $C_{F_j}$ defined in \eqref{eq:MSS}. Note that, to determine that there is no augmenting path, we have to check all alternating paths and hence the set $C_{F_j}$ is automatically generated when we check for an augmenting path. One way to achieve this is via a recursive, depth-first search algorithm \cite{Duf81,We00}. The details are given in \Cref{alg:Augmentpath}. Here, we use the notation $\cox_{\coi}, \Cof_{\coj}$ on the one hand as vertices in $\VV$ and $\VE$ and on the other hand as indices referring to those vertices. The complete algorithm, which is known as the Pantelides algorithm \cite{Pan88}, is presented in \Cref{alg:PantDAE}. 

Note that in contrast to the original algorithm \cite{Pan88}, we replace equations that are differentiated and do not add the equations to the graph. All variables occurring in the replaced equation are deleted. Thus the replaced equation has no edges and can be removed without any changes for the algorithm. This has the advantage that we do not need to keep track, which equation was obtained by differentiating another equation, and similarly for the variables. We emphasize that we only replace equations in our algorithm. If we want to use the results for numerical time-integration methods, then we can either construct a reformulation of the DAE that explicitly contains all algebraic constraints from a reduced derivative array \cite{Ste06}, where only derivatives of equations are added, which are determined by the Pantelides algorithm, or we solve the overdetermined system that results from adding all equations that are differentiated to the original system (as for instance in the dummy derivative approach \cite{MatS93} or the least-squares approach \cite{Ste06}). Note that we use a similar notation in \Cref{alg:PantDAE} as in \Cref{alg:Augmentpath} for vertices in $\VV$ and $\VE$. The vertex $\cox_{\coi,\cok}$ represents the variable $x_i^{(k)}$.

\begin{algorithm}[htb]
	\caption{Augmentpath}
	\label{alg:Augmentpath}	
	\begin{algorithmic}[1]
		\Require
		\begin{description}
			\item bipartite graph $\Cog = (\Cov_\Coe \dot{\cup} \Cov_\Cov, \Coe)$ with $\Cov_\Coe = \{\Cof_\code{1}, \dots,\Cof_\code{m}\}$ and $\Cov_\Cov = \{\cox_\code{1},\dots,\cox_\con\}$ 
			\item $\Cof_\coj \in \Cov_\Coe$ \Comment{starting vertex}
			\item $\code{assign}$ \Comment{matching}
			\item $\code{colorV}$, $\code{colorE} $ \Comment{vertices needed to be differentiated if \code{pathfound == false}}
		\end{description}
		\Ensure  \code{pathfound}, \code{assign}, \code{colorV}, \code{colorE}
		\Statex
		\State $\code{pathfound} \gets \code{false}$
		\State $\code{colorE}(\Cof_\coj)\leftarrow 1$ 
		\For {$\cox_\coi \in \Cov_\Cov$ with $\{\cox_\coi,\Cof_\coj\} \in \Coe$ and $\code{colorV}(\cox_\coi)== 0$}  \label{aug3}
		\State $\code{colorV}(\cox_\coi)\leftarrow 1$ \label{aug4}
		\State $\Cof_\cok \gets \code{assign}(\cox_\coi )$ \label{aug5}
		\If{$\Cof_\cok  == 0  $}
		\State  \code{pathfound} $\leftarrow$ \code{true}
		\Else 
		\State (\code{pathfound}, \code{assign}, \code{colorV}, \code{colorE}) $\leftarrow$ \code{Algorithm1}(\Cog, $\Cof_\cok $, \code{assign}, \code{colorV}, \code{colorE}) \label{aug10}  
		\EndIf
		\If{ \code{pathfound} $==$ \code{true}  }
		\State $\code{assign}(\cox_\coi) \leftarrow \Cof_\coj$
		\State \textbf{return} \code{pathfound}, \code{assign}, \code{colorV}, \code{colorE} 
		\EndIf  
		\EndFor 
	\end{algorithmic}
\end{algorithm}

\begin{algorithm}[htb]
	\caption{Pantelides Algorithm for \glspl{DAE}}
	\label{alg:PantDAE}	
	\begin{algorithmic}[1]
		\Require graph $\Cog = (\Cov_\Coe \dot{\cup} \Cov_\Cov, \Coe)$, $\Cov_\Coe =
			\left\{ \Cof_{\coj, \ell} \mid  \Cof_\coj^{(\ell)} \in \Cov_\Coe \right\}$, $\Cov_\Cov = 
			\left\{ \cox_{\coi, \cok} \mid  \cox_\coi^{(\cok)} \in \Cov_\Cov \right\}$		
		\Ensure graph $\Cog= (\Cov_\Coe \dot{\cup} \Cov_\Cov, \Coe)$
			\Statex
			\State $\code{assign} ( \cox_{\coi, \cok} ) \leftarrow \code{0} $ for each $\cox_{\coi, \cok} \in \Cov_\Cov $
			\For{$\cor=1,\dots,\code{M} $} 
				\State $(\cop,\coq) \leftarrow (\cor,0) $
				\Repeat
					\State\label{alg:PantDAE-lineRemoveVariables} $\tilde{\Coe} \gets \Coe \setminus \{ \{\cox_{\coi, \cok},\Cof_{\coj,\ell}\} \in \Coe \mid \cox_{\coi,\cok+ \code{s}} \in \Cov_\Cov \text{ for some } \code{s}>0  \} $
					\State $\tilde{\Cov}_\Cov \gets \Cov_\Cov \setminus   \{\cox_{\coi, \cok} \mid \cox_{\coi,\cok+ \code{s}} \in \Cov_\Cov \text{ for some } \code{s}>0  \} $
					\State $\code{colorV}(\cox_{\coi, \cok} ) \leftarrow 0 $ for each $\cox_{\coi, \cok} \in \tilde{\Cov}_\Cov $
					\State $\code{colorE}(\Cof_{\coj, \ell} ) \leftarrow 0$ for each $ \Cof_{\coj, \ell} \in \Cov_\Coe$ 
					\State (\code{pathfound}, \code{assign}, \code{colorV}, \code{colorE}) $\leftarrow$ \text{Algorithm1}($(\Cov_\Coe \dot{\cup} \tilde{\Cov}_\Cov,\tilde{\Coe})$, $\elementVE_{\cop,\coq}$, 
					\Statex \hspace{4em}\code{assign}, \code{colorV}, \code{colorE}) \Comment{apply Augmentpath} 
					\If{\code{pathfound} == \code{false}  } 	
						\State $\Cov_\Cov\gets \Cov_\Cov \cup \{\cox_{\coi, \cok+1} \mid \code{colorV}(\cox_{\coi, \cok})== 1 \}$					
						\ForEach{$(\coj,\ell)$ s.\,t.\ $\code{colorE}(\Cof_{\coj, \ell}) == 1$}
							\State $\Cof_{\coj,\ell} \gets \Cof_{\coj, \ell+1}$
							\State $\Coe \gets \Coe \cup \left\{\{\cox_{\coi, \cok+1},\Cof_{\coj, \ell}\} \mid \{\cox_{\coi, \cok},\Cof_{\coj, \ell}\}\in \Coe \right\}$
						\EndFor																
						\ForEach{$\coi$ with $\code{colorV}(\cox_{\coi, \cok} ) == 1$ and $\Cof_{\coj, \ell} == \code{assign}(\cox_{\coi, \cok} )$}
							\State  $\code{assign}(\cox_{\coi, \cok+1} ) \gets \Cof_{\coj,\ell}$
						\EndFor
					\EndIf  
				\Until{\code{pathfound} == \code{true}  }
			\EndFor	
	\end{algorithmic}
\end{algorithm}

Since \Cref{alg:PantDAE} is an iterative algorithm we have to answer the question whether \Cref{alg:PantDAE} terminates. We immediately observe that if \Cref{alg:PantDAE} terminates, then we have no exposed equation and thus each equation is matched with a highest derivative. We conclude that a necessary condition for termination of \Cref{alg:PantDAE} is that the \gls{DAE} \eqref{eq:nlDAE} is structurally nonsingular with respect to $\setvars/R_{\mathrm{equal}}$ with $\setvars$ denoting the set of variables of \eqref{eq:nlDAE} and $R_{\mathrm{equal}}$ the equivalence relation
\begin{align}
	\label{eq:eqrelEquiv}
	R_{\mathrm{equal}} \defEqual \left\{ (\elementsetvars, \tilde{\elementsetvars}) \in \setvars \times \setvars \,\left|\, \exists \xi \in \set(x),\, p,q \in \N_0 \text{ such that } \elementsetvars = \xi^{(p)},\, \tilde{\elementsetvars} = \xi^{(q)} \right.\right\}.
\end{align}
The following result shows that in fact the structural nonsingularity is also a sufficient condition.

\begin{theorem}
	\label{thm:TerminationPDAE}
	Consider the \gls{DAE} \eqref{eq:nlDAE} and let $\setvars$ denote its set of variables and $R_{\mathrm{equal}}$ the equivalence relation in \eqref{eq:eqrelEquiv}. Then \Cref{alg:PantDAE} applied to the \gls{DAE} \eqref{eq:nlDAE} terminates if and only if \eqref{eq:nlDAE} is structurally nonsingular with respect to $\eqclass{\setvars}{R_{\mathrm{equal}}}$.
\end{theorem}

\begin{proof}
	We show that the structural nonsingularity is equivalent to the termination criterion that is presented in the original paper \cite{Pan88}. To this end let us rewrite \eqref{eq:nlDAE} by splitting the variable $x$ into $\set(x) = \set(y) \dot{\cup} \set(z)$ such that 	\begin{itemize}
		\item for all $\xi \in \set(y)$ it holds that $\dot{\xi}$ appears in $F$ and 
		\item for all $\xi \in \set(z)$ it holds that $\dot{\xi}$ does not appears in $F$. 
	\end{itemize}
	Using the new variables we can rewrite \eqref{eq:nlDAE} as
	\begin{subequations}
		\label{eq:extendedDAE}
		\begin{align}
			\label{eq:DAElong}
			F(t,y,\dot{y},z) = 0. 
		\end{align}
		We introduce a coupling between $y$ and $\dot{y}$ by adding formal equations
		\begin{align}
			\label{eq:formalCoupling}
			G_i(y_i, \dot{y}_i) & = 0 \qquad\text{ for } i=1,\dots, n_y
		\end{align}	
	\end{subequations}
	for $y = [y_i]_{i = 1, \dots, n_y}$ and $G = [G_i]_{i=1,\dots, n_y}$. The system of equations \eqref{eq:extendedDAE} is called the  \defn{extended \gls{DAE}}. 
	In \cite{Pan88} it is shown that \Cref{alg:PantDAE} applied to the \gls{DAE} \eqref{eq:nlDAE} terminates if and only if the extended \gls{DAE} \eqref{eq:extendedDAE} is structurally nonsingular with respect to $\eqclass{\set(y,\dot{y},z)}{R_\mathrm{triv}}$ with $R_\mathrm{triv} \defEqual  \{(\xi, \xi) \mid \xi \in \set(y, \dot{y},z) \}$. It therefore suffices to show that structural nonsingularity of \eqref{eq:DAElong} with respect to $\eqclass{\set(y, \dot{y},z)}{R_\mathrm{equal}}$ is  equivalent to structural nonsingularity of the extended \gls{DAE} \eqref{eq:extendedDAE} with respect to $\eqclass{\set(y,\dot{y},z)}{R_\mathrm{triv}}$.
	
	Assume first that the extended \gls{DAE}  \eqref{eq:extendedDAE} is structurally nonsingular with respect to $\eqclass{\set(y,\dot{y},z)}{R_\mathrm{triv}}$. Consider an arbitrary subset of \eqref{eq:DAElong} 
	\begin{align}
		\label{eq:subsetDAElong}
		\hF(t, \ya, \yb, \ybd, \ycd, \widehat{z}) &= 0,
	\end{align}
	with $\set(\ya,\yb,\yc) \subseteq \set (y)$ and $\set(\widehat{z}) \subseteq \set(z)$. Let $\widehat{G}(\yb,\ybd) = 0$ denote the subset of \eqref{eq:formalCoupling} that corresponds to the variable $\yb$. Since the extended \gls{DAE} is structurally nonsingular, we obtain
	\begin{align*}
		\left| \set \left(  \widehat{F},\widehat{G} \right) \right| \le \left| \eqclass{\set \left( \ya, \yb, \ybd, \ycd,\widehat{z} \right)}{R_\mathrm{triv}} \right|.
	\end{align*}
	This is equivalent to 
	\begin{align*}
		\left| \set \left(  \widehat{F} \right) \right| \le \left|\set \left( \ya, \yb, \ycd,\widehat{z} \right) \right| = \left| \eqclass{\set(\ya, \yb, \ybd, \ycd, \widehat{z}) }{R_{\mathrm{equal}}}\right|
	\end{align*}	
	and thus \eqref{eq:DAElong} is structurally nonsingular with respect to $\eqclass{\setvars}{R_{\mathrm{equal}}}$.
	
	Conversely suppose that the \gls{DAE} \eqref{eq:DAElong} is structurally nonsingular with respect to $\eqclass{\set(y,\dot{y},z)}{R_\mathrm{equal}}$. A subset of \eqref{eq:extendedDAE} is given by
	\begin{equation}
		\label{eq:subsetExtendedDAE}
		\begin{aligned}
			\hF(t,\ya,\yaG,\yb,\ybd,\ybG,\ybdG,\ycd,\ycdG,\hz) &= 0,\\
			\hG(\yaG,\yadG,\ybG,\ybdG,\ycG,\ycdG,\ydG,\yddG) &= 0.
		\end{aligned}
	\end{equation}
	with $\set(\ya,\yaG,\yb,\ybG,\yc,\ycG,\ydG) \subseteq \set(y)$ and $\set(\hz)\subseteq \set(z)$.
	Since the \gls{DAE} \eqref{eq:DAElong} is structurally nonsingular with respect to $\eqclass{\set(y, \dot{y},z)}{R_\mathrm{equal}}$ we obtain 
	\begin{align*}
		\left| \set \left( \hF \right)\right| 
		\le \left| \eqclass{\set \left(  \ya, \yaG, \yb, \ybd, \ybG, \ybdG, \ycd, \ycdG, \hz  \right) }{R_{\mathrm{equal} }}  \right| 
		= \left| \set \left( \ya, \yaG, \yb, \ybG, \ycd, \ycdG, \hz \right)\right| 
	\end{align*}
	which immediately implies  
	\begin{align*}
		\left| \set\left( \hF, \hG\right)\right|
		\leq \left| \set \left( \ya, \yaG, \yadG, \yb, \ybG, \ybdG, \ycd, \ycG, \ycdG, \ydG, \yddG, \hz \right)\right| .
	\end{align*}
	We conclude that \eqref{eq:subsetExtendedDAE} is structrually nonsingular with respect to $R_{\mathrm{triv}}$, which completes the proof.
\end{proof}

Even if \Cref{alg:PantDAE} terminates this is not a guarantee that the correct number of differentiations for each equation is determined. In fact, the Pantelides algorithm relies heavily on the so-called sparsity pattern of the \gls{DAE}, i.e., each equation depends only on few variables. In particular, \Cref{alg:PantDAE} is not invariant under bijective transformations of the \gls{DAE} (cf.\ \cite[Remark~5]{SchS16a}). 

\begin{example}
	\label{ex:failurePantelidesDAE}
	The Pantelides algorithm applied to the \gls{DAE}
	\begin{equation*}
		\label{eq:PantelidesFail1}
		\dot{x}_1 + \dot{x}_2 = x_2 + f_1,\qquad 0 = x_1 + x_2 + f_2
	\end{equation*}
	determines that the second equation needs to be differentiated one time. If we however add the first equation to the second equation, then we can match the first equation with $\dot{x}_1$ and the second equation with $\dot{x}_2$ and thus neither of the equations is differentiated.
\end{example}

Also, it may be the case that the Pantelides algorithm concludes that equations need to be differentiated several times, although this is not necessary \cite{ReiMB00}. One possible reason for this is that the \gls{DAE} is close to a high-index \gls{DAE} and thus the numerical solution may be more difficult to obtain than the (potentially) small index suggests (see \cite{ReiMB00} for further details). In any case, it is possible to check that the equations are differentiated sufficiently often by applying a success check \cite{Pry01,PryNeTa15} and an algebraic criterion \cite{SchS16a}. Similar arguments apply to the relation between solvability and structural singularity.

\begin{example}
		\label{ex:regularVsStructuralSingularity}
		Consider the DAE
		\begin{displaymath}
			\begin{bmatrix}
				1 & a\\
				a & a
			\end{bmatrix}\begin{bmatrix}
				\dot{x}_1\\
				\dot{x}_2
			\end{bmatrix} = \begin{bmatrix}
				1 & a\\
				a & a
			\end{bmatrix}\begin{bmatrix}
				x_1\\
				x_2
			\end{bmatrix} + \begin{bmatrix}
				f_1\\f_2
			\end{bmatrix}.
		\end{displaymath}
		For $a=1$ the DAE is structurally nonsingular. However, solutions can only exist if $f_1-f_2\equiv 0$. On the other hand, if $a=0$, then the DAE is structurally singular. Still, the DAE possesses a solution if $f_2\equiv 0$. Note that small perturbations of the nonzero matrix coefficients and the inhomogeneities $f_1$ and $f_2$ do not affect the structural singularity but may change the solvability of the DAE.
\end{example}

\section{The Pantelides Algorithm for DDAEs}

\label{sec:pantelidesDDAE}
It is well-known that for \glspl{DDAE} of the form \eqref{eq:nlDDAE} it is not sufficient to differentiate equations to obtain solutions, but in addition, some equations need to be shifted \cite{Cam95,HaM12,HaM16,TreU19}. 
 
\begin{example}
	\label{ex:shiftAndDifferentiate}
	Consider the \gls{DDAE}
	\begin{equation}
	\label{eq:Ex2}
		\dot{x}_1 =  f_1,\qquad 0 = x_1 - \shift{x_2}{-\tau} + f_2. 
	\end{equation}
	If we want to construct a solution for the corresponding initial value problem with the method of steps \cite{BelZ03}, then in the interval $[0,\tau)$ we have to solve the DAE
	\begin{equation}
		\label{eq:Ex2:2}
		\dot{x}_1 = f_1, \quad 0 = x_1 + \tilde{f}_2
	\end{equation}
	with $\tilde{f}_2 = f_2 + \varphi_2(t-\tau)$. Clearly, this DAE is not regular (it does not depend on $x_2$ and the two equations might contradict each other). If we however add the first equation to the differentiated second equation and shift the resulting equation, we obtain the DDAE
	\begin{displaymath}
		\dot{x}_1 = f_1, \qquad \dot{x}_2 = \shift{f_2}{\tau} + \shift{\dot{f}_1}{\tau}
	\end{displaymath}
	and it is easy to see that for any consistent initial value the DAE that we have to solve in the interval $[0,\tau)$ is uniquely solvable.
\end{example}

As a direct consequence, a structural analysis for \glspl{DDAE} should reveal which equations need to be shifted and which equations need to be differentiated. We thus have to understand the similarities and differences between differentiating and shifting from a structural point of view.

\subsection{Differentiation vs.\ Shifting}
\label{subsec:diffvsshift}
To understand the effect of the two operators $\frac{\mathrm{d}}{\mathrm{d}t}$ and $\shift{\,}{-\tau}$ from a structural point of view, we start with the following simple example.

\begin{example}
	\label{ex:LTISystem}
	Let $N\in\mathbb{R}^{n\times n}$ be a nilpotent matrix with $N^\nu = 0$ and $N^{\nu-1}\neq 0$ for some $\nu\in\mathbb{N}$ and consider the two equations
	\begin{subequations}
		\begin{align}
			\label{eq:LTIDAE} N\frac{\mathrm{d}}{\mathrm{d}t}x &= x + f,\\
			\label{eq:LTIDE} Ny &= \shift{y}{-\tau} + g,
		\end{align}
		which can be understood as special cases of the Weierstraß canonical form \cite{Gan59b} for a regular matrix pencil. It is easy to see that the solutions are of the form
		\begin{equation}
			\label{eq:solutionNilpotentLTI}
			x = -\sum_{i=0}^{\nu-1} N^i \left(\frac{\mathrm{d}}{\mathrm{d}t}\right)^i f\qquad\text{and}\qquad
			y = -\sum_{i=0}^{\nu-1} N^i \left(\shift{\,}{\tau}\right)^i \shift{g}{\tau}.
		\end{equation}
	\end{subequations}
	Apart from the operator applied to the inhomogeneity both solutions are the same. Thus the variable $\dot{x}$ has in the \gls{DAE} the same role as the variable $y$ in the difference equation \eqref{eq:LTIDE}. In other words, in the \gls{DAE}, the variable with the highest derivative is $\dot{x}$, while in the delay equation, the variable $y = \shift{y}{0}$ is the variable with the highest shift, in the sense that $0 > -1$. 
\end{example}

Motivated by \Cref{ex:LTISystem} we can use the Pantelides algorithm (\Cref{alg:PantDAE}) for difference equations by replacing differentiation by shifting in \Cref{alg:PantDAE}. 
In a difference equation we want to assign each equation to a different variable with \defn{highest shift}, meaning that $\shift{x}{k\tau}$ with $k\geq 0$ occurs in some equation but $\shift{x}{(k+\ell) \tau}$ for $\ell>0$ does not. Note that by our definition the variable $\shift{x}{-\tau}$ can never be a variable with \defn{highest shift}, since this would imply that we can solve directly for this variable without shifting. This can  also be seen in the solution formula \eqref{eq:solutionNilpotentLTI} for the difference equation, which requires an additional shift of the inhomogeneity $g$. 

\begin{example}
	\label{ex:LTI1}
	The graph of the scalar difference equation
	\begin{equation}
		\label{eq:exDDE}
		0 =\shift{x_1}{-\tau} + f_1 
	\end{equation}		
	 is visualized in \Cref{fig:structual_viewpoint_shift_dif1}.
		\begin{figure}[htb]
			\hspace{1ex}
			\begin{subfigure}[t]{0.48\textwidth}
				\centering
				\begin{tikzpicture}
					\enode[try]{1}{$F_1$}
					\vnode[low]{1}{$\shift{x_1}{-\tau}$}{}
					\edge[low]{e1}{v1}
				\end{tikzpicture}	
				\caption{Initial graph for \eqref{eq:exDDE}}
				\label{fig:structual_viewpoint_shift_dif1}	
			\end{subfigure}
			\qquad
			\begin{subfigure}[t]{0.48\textwidth}
				\centering
				\begin{tikzpicture}
					\enode{1}{$\shift{F_1}{\tau}$}
					\vnode{1}{$x_1$}{}
					\vnode[low]{2}{$\shift{x_1}{-\tau}$}{}
					\edge{e1}{v1}
				\end{tikzpicture}
				\caption{Graph for \eqref{eq:exDDE} after shifting $F_1$}	
				\label{fig:structual_viewpoint_shift_dif2}	
			\end{subfigure}
			\caption{Graphs obtained by applying \Cref{alg:PantDAE} to \Cref{ex:LTI1}}
		\end{figure}
 	By replacing differentiating with shifting in \Cref{alg:PantDAE}, we shift $F_1$ once (cf.\ \Cref{fig:structual_viewpoint_shift_dif2}). Note that there is no edge between $\shift{F_1}{\tau}$ and $\shift{x_1}{-\tau}$, since there is no chain rule for the shift operator.
\end{example}

\subsection{The Shifting and Differentiation Graph}
\label{subsec:ShiftAndDifGraph}
As outlined in \Cref{ex:shiftAndDifferentiate}, some equations in the \gls{DDAE} \eqref{eq:nlDDAE} need to be shifted and some need to be differentiated. Motivated by the observation (cf.\ \cref{subsec:diffvsshift}) that shifting and differentiating are quite similar from a structural point of view, we construct two graphs: One to determine which equations need to be shifted and one to determine which equations need to be differentiated. In the shifting graph, we do not want to match equations with delayed variables, since we cannot (directly) solve for these variables. More precisely, we do not want to match equations with variables that are not highest shifts and thus delete these variables in the shifting graph. On the other hand, we need to make sure that we do not match one equation with a variable and another equation with the derivative of that variable. This is illustrated in the following example.

\begin{example}
	\label{ex:shiftingGraph}
	Consider again the \gls{DDAE} \eqref{eq:Ex2}.  Deleting the variable $\shift{x_2}{\tau}$ results in the graph \Cref{fig:DDAEEx1G1} with maximal matching $M = \{ \{F_1, \{x_1\}\}, \{F_2, \{\dot{x}_1\}\}\}$. Note that the maximal matching is possible, since we match both equations with $x_1$ respectively its derivative $\dot{x}_1$. 
\begin{figure}[htb]
	\centering
	\begin{subfigure}[b]{0.47\textwidth}
		\centering 
		\begin{tikzpicture}
		\enode{1}{$F_1$}
		\enode{2}{$F_2$}
		\vnode{1}{$x_1$}{}
		\vnode{2}{$\dot{x}_1$}{}
		\vnode[low]{3}{$\shift{x_2}{-\tau}$}{}
		\edge[assign]{e1}{v2}
		\edge[assign]{e2}{v1}
		\edge[low]{e2}{v3}
		\end{tikzpicture}
		\caption{without grouping}	
		\label{fig:DDAEEx1G1}			
	\end{subfigure}
	\hfill
	\begin{subfigure}[b]{0.47\textwidth}
		\centering
		\begin{tikzpicture}
		\enode[try]{1}{$F_1$}
		\enode[try]{2}{$F_2$}
		\vnode{1}{$x_1, \dot{x}_1$}{}
		\vnode[low]{2}{$\shift{x_2}{-\tau}$}{}
		\edge[assign]{e1}{v1}
		\edge{e2}{v1}
		\edge[low]{e2}{v2}	
		\end{tikzpicture}
		\caption{with grouping}	
		\label{fig:DDAEEx1G2}
	\end{subfigure}
	\caption{Different shifting graphs for \Cref{ex:shiftingGraph}}
\end{figure}
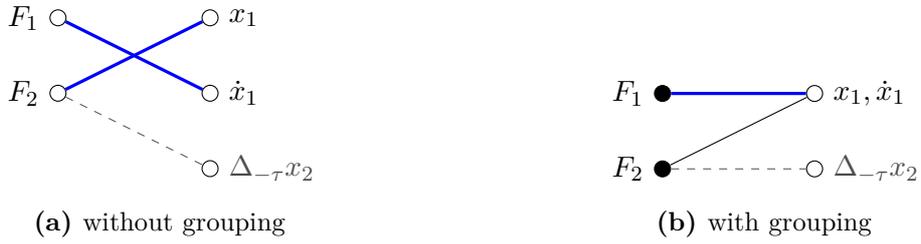
If we instead merge the variables $x_1$ and $\dot{x}_1$ into a single vertex we obtain the graph in \Cref{fig:DDAEEx1G2}, which indicates that we have to shift some of the equations (in agreement with the observation in \Cref{ex:shiftAndDifferentiate}. 
\end{example}

To ensure that we do not match different equations with the same variable (with different differentiation levels), we need to group the corresponding variables as in \Cref{ex:shiftingGraph}. To merge variables of different differentiation levels we use the equivalence classes introduced in the graph for a \gls{DAE} in \Cref{def:graphDAE}. Instead of imposing the trivial relation, we introduce a new relation. To formalize this, consider the equation
\begin{equation}
	\label{eq:DDAEshiftDiffVars}
	F(t,\shift{x}{-\tau},\shift{\dot{x}}{-\tau},\ldots,\shift{x^{(\rho)}}{-\tau},\ldots,\shift{x}{\kappa\tau},\shift{\dot{x}}{\kappa\tau},\ldots,\shift{x^{(\rho)}}{\kappa\tau}) = 0.
\end{equation}
which may be understood as a generalization of \eqref{eq:nlDDAE} that appears after shifting and/or differentiation of some of the equations.

\begin{definition}
	\label{def:shiftDiffGraph}
	Consider the \gls{DDAE} \eqref{eq:DDAEshiftDiffVars} and let $\setvars$ denote the set of variables for \eqref{eq:DDAEshiftDiffVars} (with $\allvars^T = \begin{bmatrix}
		{\shift{x}{-\tau}}^T & \dots & {\shift{x^{(\rho)}}{\kappa\tau}}^T
	\end{bmatrix}$). 
	\begin{itemize}
		\item The variables $\elementsetvars,\widetilde{\elementsetvars}\in\setvars$ are called \defn{shifting similar} if there exists $\xi\in\set(x)$ and integers $k\in\mathbb{Z}$, $p,q\in\mathbb{N}_0$ such that $\elementsetvars = \shift{\xi^{(p)}}{k\tau}$ and $\widetilde{\elementsetvars} = \shift{\xi^{(q)}}{k\tau}$. Shifting similarity introduces an equivalence relation on $\setvars$, which we denote by $\Rs$. The graph of \eqref{eq:DDAEshiftDiffVars} over $\Rs$ is called \defn{shifting graph}.
		\item The variables $\elementsetvars,\widetilde{\elementsetvars}\in\setvars$ are called \defn{differential similar} if there exists $\xi\in\set(x)$ and integers $p,k,\ell\in\mathbb{N}_0$ such that $\elementsetvars = \shift{\xi^{(p)}}{k\tau}$ and $\widetilde{\elementsetvars} = \shift{\xi^{(p)}}{\ell\tau}$. Differential similarity introduces an equivalence relation on $\widetilde{\setvars}$ with
			\begin{displaymath}
			\label{eq:varExcludingNegTau}
				\widetilde{\setvars} = \{\elementsetvars \in \setvars \mid \text{there exists } x\in\setvars, k,p\in\mathbb{N}_0 \text{ such that } \elementsetvars = \shift{x^{(p)}}{k\tau}\},
			\end{displaymath}
			which we denote by $\Rd$. The graph of \eqref{eq:DDAEshiftDiffVars} over $\Rd$ is called \defn{differentiation graph}.
		\item The graph of \eqref{eq:DDAEshiftDiffVars} over the trivial equivalence relation $R \defEqual \{(\theta, \theta) \mid \theta \in \Theta \}$ is called \defn{graph of the \gls{DDAE}}. 
	\end{itemize}
\end{definition}

\begin{example}
	The shifting graph for the \gls{DDAE} \eqref{eq:Ex2} is given in \Cref{fig:DDAEEx1G2}. For the differentiation graph we have
\begin{displaymath}
\widetilde{\setvars} = \left\{\{x_1\},\{\dot{x}_1\} \right\},
\end{displaymath}
such that the differentiation graph is given in \Cref{fig:DDAEEx1diffGraph}. Note that $\{x_2\}\not\in\widetilde{\setvars}$.
\begin{figure}[htb]
	\centering
	\begin{tikzpicture}
	\enode{1}{$F_1$}
	\enode[try]{2}{$F_2$}
	\vnode[low]{1}{$x_1$}{}
	\vnode{2}{$\dot{x}_1$}{}
	\edge[assign]{e1}{v2}
	\edge[low]{e2}{v1}
	\end{tikzpicture}
	\caption{Differentiation graph for the \gls{DDAE} \eqref{eq:Ex2}}	
	\label{fig:DDAEEx1diffGraph}
\end{figure}
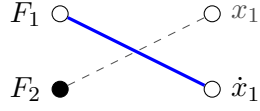
\end{example}
 
A question that arises is how differentiation of equations affects the shifting graph and how shifting affects the differentiation graph (cf.\ \cite{HAMS14} for an example where the order of shifting and differentiation results in different smoothness requirements). We immediately observe that the differentiation graph is affected by shifting, see the following example.

\begin{example}
	The differentiation graph for the scalar \gls{DDAE}
	\begin{equation}
		\label{eq:shiftDDAE}
		0 = \shift{x_1}{-\tau} + f_1
	\end{equation}
	does not contain any variable vertices and thus in particular no maximal matching. If we shift \eqref{eq:shiftDDAE} we obtain
	\begin{displaymath}
		0 = x_1 + \shift{f_1}{\tau},
	\end{displaymath}
	such that the differentiation graph is given by $V = \{F_1,\{x_1\}\}$ and $E = \{\{F_1,\{x_1\}\}\}$.
\end{example}

To understand the effect of differentiation on the shifting graph, let us define the system of equations
\begin{equation}
	\label{eq:DDAEdifferentiated}
	\replaceSys{F}(t,\shift{x}{-\tau},\ldots,\shift{x^{(\rho)}}{-\tau},\shift{x^{(\rho+1)}}{-\tau},\ldots,\shift{x}{\kappa\tau},\shift{\dot{x}}{\kappa\tau},\ldots,\shift{x^{(\rho+1)}}{\kappa\tau}) = 0.
\end{equation}
by replacing the $i$th equation ($i\in\{1,\ldots,M\}$) in \eqref{eq:DDAEshiftDiffVars} with its total derivative with respect to $t$. The shifting graphs of \eqref{eq:DDAEshiftDiffVars} and \eqref{eq:DDAEdifferentiated} are denoted by $G^\mathrm{s} \defEqual (\VEs \dot{\cup} \VVs, E^\mathrm{s})$ and $\replaceSys{G}^\mathrm{s} \defEqual (\replaceSys{V}_\mathrm{E}^\mathrm{s} \dot{\cup} \replaceSys{V}_\mathrm{V}^\mathrm{s}, \replaceSys{E}^\mathrm{s})$, respectively. 
Let us define the bijection $\phi\colon \VEs \dot{\cup} \VVs \to \replaceSys{V}_\mathrm{E}^\mathrm{s} \dot{\cup} \replaceSys{V}_\mathrm{V}^\mathrm{s}$ via
\begin{displaymath}
	\phi(v) = \begin{cases}
		v \cup \{\dot{w} \mid w \in v \}, & \text{for } v \in \VVs, \\
		v, & \text{for } v \in \VEs \text{ and } v \neq F_i, \\
		\dot{v}, & \text{for } v\in \VEs \text{ and } v = F_i.
	\end{cases}
\end{displaymath}
Then it is easy to see, that $\{u,v\}\in E^{\mathrm{s}}$ if and only if $\{\phi(u),\phi(v)\}\in\replaceSys{E}^{\mathrm{s}}$, that is the graphs $G^{\mathrm{s}}$ and $\replaceSys{G}^{\mathrm{s}}$ are isomorphic \cite[p.~21]{Jun13}. In particular, we have shown the following result.

\begin{prop}
	\label{prop:PDDAEdiffDoesntEffectShift}
	Consider the \gls{DDAE} \eqref{eq:DDAEshiftDiffVars}. Differentiation of an equation in \eqref{eq:DDAEshiftDiffVars} does not affect the shifting graph of \eqref{eq:DDAEshiftDiffVars}.
\end{prop}

Since shifting affects the differentiation graph but differentiation does not affect the shifting graph, we conclude that we first determine which equations need to be shifted and afterwards determine which equations should be differentiated.

\begin{remark}
	\label{rem:shiftingRequiresDifferentiation}
	The strategy to shift first can be understood as replacing the \gls{DDAE} \eqref{eq:nlDDAE} with a transformed \gls{DDAE} that can be solved with the method of steps. For linear time-invariant \glspl{DDAE} such a strategy was investigated in \cite{Cam95,TreU19} in a polynomial matrix framework. There the authors use a sequence of row compressions to determine which equations need to be shifted. It should be noted that the row-compression requires differentiation of some of the equations to enable the shift afterward. Fore more details we refer to \cite{Cam95,TreU19} and the upcoming subsection.
\end{remark}

\subsection{Differentiating during the Shifting Step}
\label{subsec:DiffInShift}
If we find (for instance via \Cref{alg:PantDAE}) an exposed equation in the shifting graph, we conclude that this equation and all equations that are connected with this equation by an alternating path need to be shifted. However, simply shifting all these equations may not be sufficient, since the connection may exist only implicitly via the equivalence class. We illustrate this with the following example.

\begin{example}
	The shifting graph of \Cref{ex:shiftingGraph}, i.e., the shifting graph for the \gls{DDAE} \eqref{eq:Ex2}, is given in \Cref{fig:DDAEEx1G2}. If we match the first equation $F_1$ with the equivalence class $\{x_1,\dot{x}_1\}$ then $F_2$ is an exposed equation that is connected with $F_1$ via $\{x_1,\dot{x}_1\}$. Accordingly, we have to shift both equations. If we check the graph of \eqref{eq:Ex2} (presented in \Cref{fig:DDAEEx1G1}), then we observe that there is no path from $F_1$ to $F_2$, i.e., we only obtain a path if we differentiate $F_2$.
\end{example}

To resolve an (possibly) implicit connection, we have to ensure that there also exists a path in the graph of the \gls{DDAE} and not only in the shifting graph. This can be achieved by differentiating the equation that does not depend on the highest derivative in the equivalence class.
To ensure that all equations which need to be shifted are explicitly connected we do not need to resolve all possible implicit connections but only those that ensure a direct path in the graph of the \gls{DDAE}. Consider the following example. 

\begin{example}
	\label{ex:LPnotSolvable}
	Consider the \gls{DDAE}	
	\begin{equation}
	\label{eq:LPnotSolvable}
	\dot{x}_1 = f_1,\qquad \dot{x}_1 = x_2 + f_2,\qquad 0 = x_1 + x_2 + \shift{x_3}{-\tau} + f_3,
	\end{equation}
	which can be transformed to the \gls{DDAE} \eqref{eq:Ex2} by inserting $F_1$ in $F_2$ and then inserting the result into $F_3$ yielding
	\begin{align*}
		\dot{x}_1 = f_1, \quad x_1 + \xnsd{3}{-1}{0} + f_1 - f_2 + f_3. 
	\end{align*} 
	Constructing the shifting graph of \eqref{eq:LPnotSolvable} and applying the shifting step to it, yields that the equation $F_3$ is exposed  and connected through alternating paths to $F_1$ and $F_2$. Thus all three equations need to be shifted and we need to find a direct connection in the graph of the \gls{DDAE}. The possible connections of these equations are visualized in  \Cref{fig:LPnotSolvable2} and \Cref{fig:LPnotSolvable3} by red edges. 
\end{example} 
\begin{figure}
	\begin{subfigure}[t]{0.30\textwidth}
		\centering
		\begin{tikzpicture}
		\enode[try]{1}{$F_1$}
		\enode[try]{2}{$F_2$}
		\enode[try]{3}{$F_3$}
		\vnode{1}{$x_1, \dot{x}_1$}{}
		\vnode{2}{$x_2$}{}
		\vnode[low]{3}{$\shift{x_3}{-\tau}$}{}
		\edge[assign]{e1}{v1}
		\edge{e2}{v1}
		\edge{e3}{v1}
		\edge[assign]{e2}{v2}
		\edge{e3}{v2}
		\edge[low]{e3}{v3}
		\end{tikzpicture}
		\caption{The shifting graph}	
		\label{fig:LPnotSolvable}	
	\end{subfigure}
	\,
	\begin{subfigure}[t]{0.30\textwidth}
	\centering
	\begin{tikzpicture}
	\enode[try]{1}{$F_1$}
	\enode[try]{2}{$F_2$}
	\enode[try]{3}{$F_3$}
	\vnode{1}{$x_1, \dot{x}_1$}{}
	\vnode{2}{$x_2$}{}
	\vnode[low]{3}{$\shift{x_3}{-\tau}$}{}
	\path
	(de1) edge[color = red, very thick] (dv1)
	(de2) edge[color = red, very thick] (dv1)
	(de2) edge[color = red, very thick] (dv2)
	(de3) edge (dv1)
	(de3) edge[color = red, very thick] (dv2)
	(de3) edge[verylightgray, dashed, very thin] (dv3);
	\end{tikzpicture}
	\caption{first possible connection}	
	\label{fig:LPnotSolvable2}
	\end{subfigure}
	\,
\begin{subfigure}[t]{0.30\textwidth}
	\centering
	\begin{tikzpicture}
	\enode[try]{1}{$F_1$}
	\enode[try]{2}{$F_2$}
	\enode[try]{3}{$F_3$}
	\vnode{1}{$x_1, \dot{x}_1$}{}
	\vnode{2}{$x_2$}{}
	\vnode[low]{3}{$\shift{x_3}{-\tau}$}{}
	\path
	(de1) edge[color = red, very thick] (dv1)
	(de2) edge (dv1)
	(de3) edge[color = red, very thick] (dv1)
	(de2) edge[color = red, very thick] (dv2)
	(de3) edge[color = red, very thick] (dv2)
	(de3) edge[verylightgray, dashed, very thin] (dv3);
	\end{tikzpicture}
	\caption{second possible connection}	
	\label{fig:LPnotSolvable3}
\end{subfigure}
	\caption{The shifting graph of \eqref{ex:LPnotSolvable} and possible connections}
\end{figure}
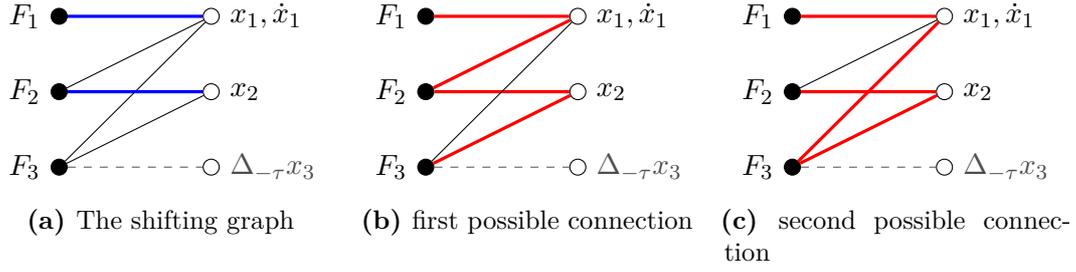	

We need to identify all possible connections. Given a matching $\mathcal{M}$, suppose that the equation $F_j$ is exposed and and the set $C_{F_j}$ as defined in \eqref{eq:MSS} is given by
\begin{align*}
		C_{F_j} \defEqual \left\{F_k \in \VE \mid \text{there exists an alternating path between } F_j \text { and } F_k \text{ in } G \right\}. 
\end{align*}
We define a \defn{connection for} $F_j$ to be a set of connected alternating paths $(F_i,v_k, F_\ell)$ with  $(v_k, F_\ell) \in \mathcal{M}$, $(F_i, v_k) \in E \setminus \mathcal{M} $ such that all  $(v_k, F_\ell) \in \mathcal{M}$ occur exactly once.
\begin{example}
	For \Cref{ex:LPnotSolvable} we immediately observe that there are two possible connections, which are displayed in \Cref{fig:LPnotSolvable2,fig:LPnotSolvable3}. Note that the connection in \Cref{fig:LPnotSolvable2} is a single alternating path, while the connection in \Cref{fig:LPnotSolvable3} consists of two alternating paths.
\end{example}

For each connection $P$ we need to decide which equations need to be differentiated. Denote by $F_{i_1}$ the exposed equation and by $F_{i_k}$ for $k=2,\ldots,K$ the equations in $C_{F_{i_1}}$. Let $\nu_{k}$ denote the number how often we have to differentiate the $i_k$th equation. For each path 
$(F_{i_{k_j}},v_j,F_{i_{k_{j+1}}}) \in P$ ($j=1,\ldots,J$) with $v_j\in \VVs$ define $b_j\in\mathbb{Z}$ to be the number how many more times we have to differentiate $F_{i_{k_j}}$ than $F_{i_{k_{j+1}}}$ such that there exists a path to a highest derivative of both equations in the graph of the \gls{DDAE}. This results in the linear system
\begin{equation}
	\label{eq:linearSystemLP}
	\nu_{k_j} - \nu_{k_{j+1}} = b_j\qquad j=1,\ldots,J.
\end{equation}
which we compactly write as $A\nu = b$ with $\nu = [\nu_k]\in\mathbb{N}_0^K$, $b = [b_j]\in\mathbb{Z}^J$ and $A\in\mathbb{Z}^{J\times K}$. 
Since we do not want to differentiate equations more than necessary, we have to solve the linear integer program
\begin{equation}
	\label{eq:LPDifInShift}
	\begin{aligned}
		&\min \sum_{k=1}^K \nu_k,\\
		&\text{such that\ }  A\nu = b,\quad \nu\in\mathbb{N}_0.
	\end{aligned}
\end{equation}

\begin{example}
	\label{ex:diffDuringShift}
	For the \gls{DDAE} \eqref{eq:LPnotSolvable} the first connection given in \Cref{fig:LPnotSolvable2} results in the linear program 
	\begin{displaymath}
	\min \nu_1 + \nu_2 + \nu_3,\quad \text{s.t. } 
		\begin{bmatrix}
		-1 & 1 & 0\\
		0 & -1 & 1
		\end{bmatrix} \nu =  \begin{bmatrix}
		0\\0
		\end{bmatrix},\quad \nu \in \N_0
	\end{displaymath}
	with the understanding that $\nu_k$ determines how often we have to differentiate the $F_k$th equation. It is easy to see that the 0 solution $\nu_1 = \nu_2 = \nu_3 = 0$ solves \eqref{eq:LPDifInShift}, i.\,e., we do not have to differentiate any equation. The other connection \Cref{fig:LPnotSolvable3} results in the linear integer program \eqref{eq:LPDifInShift} given via
	\begin{displaymath}
	A = \begin{bmatrix}
	-1 & 0 & 1 \\
	0 & -1 & 1
	\end{bmatrix}\qquad\text{and}\qquad b = \begin{bmatrix}
	1\\0
	\end{bmatrix}
	\end{displaymath}	
	with the unique solution $\nu_1 = 0$ and $\nu_2= \nu_3 = 1$. Hence, we have to differentiate $F_2$ and $F_3$ to resolve the connection.
\end{example}

A relaxation of the linear integer program \eqref{eq:LPDifInShift} is the linear program
\begin{equation}
	\label{eq:LPDifInShiftEquiv}
	\begin{aligned}
		&\min \sum_{k=1}^K \nu_k,\\
		&\text{such that\ }  A\nu = b,\quad \nu\geq 0.
	\end{aligned}
\end{equation}
In fact, we have the following result.

\begin{theorem}
	\label{thm:linearIntegerProgram}\ 
	\begin{enumerate}
		\item\label{it:linearProgram} The linear program \eqref{eq:LPDifInShiftEquiv} has a unique solution.  
		\item The linear integer program \eqref{eq:LPDifInShift} is solvable if and only if the linear program \eqref{eq:LPDifInShiftEquiv} is solvable. In this case, the minimizing vector $\nu^\star$ of \eqref{eq:LPDifInShiftEquiv} is the unique minimizer of \eqref{eq:LPDifInShift}.
	\end{enumerate}
\end{theorem}

\begin{proof}
	\ 
	\begin{enumerate}
		\item Recall that $J$ is the number of equations in the linear system $A\nu = b$ and $K$ is the number of unknowns. We immediately observe $J = K-1$ and $\rank(A) = K-1$ (the equations $F_{i_k}$ form a connected graph). Thus $A$ has full row rank and $A \nu = b$ is solvable for each $b$. 
		Thus there exists $\widehat{\nu} = [\widehat{\nu}_k]\in\R^K$ such that $A\widehat{\nu} = b$. Define $\alpha^\star = -\min \{\widehat{\nu}_k \mid k=1,\ldots,K\}$. Let $e \defEqual \begin{bmatrix}
				1 & \dots & 1\end{bmatrix}^T\in\R^K$. Then the feasible set of \eqref{eq:LPDifInShiftEquiv} is
		\begin{equation}
			\label{eq:feasibleSetLPDifInShiftEquiv}
			\{\widehat{\nu} + \alpha e \mid \alpha \geq \alpha^\star\}.
		\end{equation}
		We conclude that $\nu^\star = \widehat{\nu} + \alpha^\star e$ is the unique minimizer of \eqref{eq:LPDifInShiftEquiv}.
		\item Using \ref{it:linearProgram} it suffices to show that \eqref{eq:LPDifInShift} is solvable whenever \eqref{eq:LPDifInShiftEquiv} is solvable. Let $\nu^\star = [\nu_k^\star]$ denote the unique minimizer of \eqref{eq:LPDifInShiftEquiv}. Then the proof of \ref{it:linearProgram} implies $\min\{\nu_k^\star \mid k=1,\ldots,K\} = 0$. Since the equations \eqref{eq:linearSystemLP} imply that the difference between two entries of $\nu^\star$ is an integer, we conclude $\nu^\star\in\mathbb{N}_0^K$. The result follows from the observation that the feasible set of \eqref{eq:LPDifInShift} is given by
		\begin{displaymath}
			\{\nu^\star + \alpha e \mid \alpha \in\mathbb{N}_0\}.\qedhere
		\end{displaymath}
	\end{enumerate}
\end{proof}

\Cref{thm:linearIntegerProgram} implies that we can compute the solution (provided it exists) of the linear integer program \eqref{eq:LPDifInShift} with standard methods such as the simplex algorithm or interior-point methods \cite{NeMe65, Ka84}. Alternatively, we can exploit the structure of the feasible set \eqref{eq:feasibleSetLPDifInShiftEquiv} and simply compute one solution of the linear system $A\nu = b$ and shift the solution as described in the proof of \Cref{thm:linearIntegerProgram} \ref{it:linearProgram}.

In summary, we need to solve for each connection one linear program, which is always solvable. If we take of all solutions, the maximal number of differentiations for each equation, all implicit connections in the shifting graph are resolved. This is due to the fact that although we keep only the highest derivative/shift of an equation in our graph, one later uses all intermediate equations as well. 

\begin{example}
\label{ex:LPNotSolvableContinued} 
	Continuing with \Cref{ex:diffDuringShift} we have to differentiate $F_2$ and $F_3$ once. This immediately implies that both connections $\{(F_2, \{x_1, \dot{x}_1\}, F_1), (F_3, \{x_2\}, F_2 ) \}$ and $\{(\dot{F}_3, \{x_1, \dot{x}_1\}, F_1), (\dot{F}_3, \{x_2\}, \dot{F}_2 ) \}$ are direct.  
\end{example}
The complete algorithm is stated in \Cref{alg:DiffDuringShift}. Even though we perform differentiations during the shifting step, the differentiations do not influence the shifting graph (cf. \Cref{prop:PDDAEdiffDoesntEffectShift}). Thus we do not update the shifting graph in the differentiation step but only the graph of the \gls{DDAE} which is used to identify the number of needed differentiations. 

\begin{algorithm}
	\caption{The differentiation during the shifting step}\label{alg:DiffDuringShift}	
	\begin{algorithmic}[1]
		\Require  graph $\code{G}=  (\Cov_\Coe \dot{\cup} \Cov_\Cov, \Coe )$, shifting graph $\code{G}^\code{s} = (\Cov_\Coe^\code{s} \dot{\cup} \Cov_\Cov^\code{s}, \Coe^\code{s} )$, $\code{colorV}$, $\code{colorE}$, $\code{assign}$, exposed equation $\code{F}_\coj \in \Cov_\Coe$ 
		\Ensure graph $\code{G}$
		\Statex
		\State Find all connections $\Cop^\ell$ for $\Cof_\coj$ 
		\For{each connection $\Cop^\ell$}
			\State Construct and solve the LP \eqref{eq:LPDifInShift} with solution $\nu^\ell$  
		\EndFor 
		\State $\code{colorEdif}(\Cof_\coi) \gets \arg\max_{\ell}  \nu^\ell_\coi$ for all $\coi$ \Comment{number of differentiations for each equation}
		\State $\Cog \gets \code{Algorithm6}(\Cog, \code{dif}, \code{true}, [], \code{colorEdif})$ 
	\end{algorithmic}
\end{algorithm}

\subsection{Trimmed linearization}
\label{subsec:trimmedLinearization}

The differentiation during the shifting step yields a graph of the \gls{DDAE}, which might contain higher-order derivatives. The Pantelides algorithm is only applicable to equations that contain derivatives up to first-order, namely first-order systems. Thus, we need to reformulate the \gls{DDAE} as a first-order system. Since we summarize in the shifting graph all variables which only differ by their order of differentiation in the same equivalence class, it is sufficient to reformulate the \gls{DDAE} after the shifting step. 

The Pantelides algorithm is based on the information which variable appears in which equation. It, therefore, suffices to introduce new variables for variables that are differentiated more than once.  This approach is similar to the first-order formulation of multi-body systems, where only new variables for the velocities are introduced but not for the derivative of the Lagrange multiplier. 

\begin{remark}
	The process of reformulating a higher-order system to a first-order system is called linearization in the literature. The variant that we use here, where only variables that actually appear in the equations are replaced, is referred to as a trimmed linearization \cite{MeS06, Sc11}. If state transformations are allowed, then a minimal number of new variables can be introduced \cite{Sc11}, which in turn minimizes smoothness requirements and benefits numerical methods. From a structural point of view, such a state transformation is, however, not feasible.
\end{remark}

After the shifting step, each equation in the shifting graph is matched to a highest shifted variable. Thus we need to ensure that this property is still fulfilled after performing trimmed linearization. 

\begin{example}
	Applying the shifting step introduced in \Cref{subsec:DiffInShift} to the \gls{DDAE}
	\begin{align*}
		x_1 \xnsd{1}{0}{1} = \xnsd{2}{-1}{0} + f_1, \quad \xnsd{2}{0}{1} = \xnsd{3}{-1}{0} + f_2, \quad 0 = \xnsd{1}{-1}{0} + f_3 
	\end{align*}
	in Hessenberg form, we obtain the shifted and differentiated  \gls{DDAE}	
	\begin{align*}
		\xnsd{1}{2}{0} \xnsd{1}{2}{2} + \xnsd{1}{2}{1}^2 & = \xnsd{2}{1}{1} + \shift{\dot{f}_1}{ 2 \tau}, \\ 
		\xnsd{2}{1}{1} & = \xnsd{3}{0}{0} + \shift{f_2}{\tau}, \\ 
		0 & = \xnsd{1}{2}{2} + \shift{\ddot{f}_3}{3 \tau} 
	\end{align*}
	which contains a second derivative of $x_1$. To perform the differentiation step, we need to reformulate this \gls{DDAE} as first order system by introducing 
	\begin{align}
		\label{eq:notShifted}
		y_1 = \dot{x}_1. 
	\end{align}	
	If we replace variables accordingly we obtain the first order \gls{DDAE}
	\begin{align}
		\label{eq:TrimmedLinearizationEx1}
		\begin{aligned}
					\xnsd{1}{2}{0} \ynsd{1}{2}{1} + \ynsd{1}{2}{0}^2 & = \xnsd{2}{1}{1} + \shift{\dot{f}_1}{ 2 \tau}, \\ 
			\xnsd{2}{1}{1} & = \xnsd{3}{0}{0} + \shift{f_2}{\tau}, \\ 
			0 & = \ynsd{1}{2}{1} + \shift{\ddot{f}_3}{3 \tau},
		\end{aligned}
	\end{align}
together with \eqref{eq:notShifted}. No variable that appears in \eqref{eq:notShifted} is a highest shift and consequently we cannot assign in the shifting graph each equation to an equivalence class. If we shift \eqref{eq:notShifted} twice, the corresponding shifting graph contains a maximal matching. 
\end{example}

The following theorem provides a first order system such that applying the shifting step to it yields no shifts.  

\begin{theorem}
	\label{thm:trimmedLinearization}
	Consider a \gls{DDAE} that depends on variables that are differentiated more than once, i.e., there exists $i\in\{1,\ldots,\stateDim\}$ such that
	\begin{displaymath}
		q_i \defEqual \max \left\{ \ell\,\left|\,\shift{x_i^{(\ell)}}{k \tau},\, k \ge 0\right.\right\} \geq 2,
	\end{displaymath}
	and assume that there exists a maximal matching in the shifting graph. Define 
	\begin{displaymath}
		\ell_p \defEqual \max\{k \mid \xnsd{i}{k}{q} \text{ appears in some equation for }  q \ge p  \}.
	\end{displaymath}
	Replace for $1 \le p < q_i$ the variable $\xnsd{i}{k}{p}$ by $\ynsd{i_p}{k}{0}$ and $\xnsd{i}{k}{q_i}$ by $\ynsd{y_{q_i - 1}}{k}{1}$. 
	We add the equations 
	\begin{align*}
		\xnsd{i}{\ell_1}{1} & = \ynsd{i_1}{\ell_1}{0} \\
		\ynsd{i_1}{\ell_2}{1} & = \ynsd{i_2}{\ell_2}{0} \\
		& \vdots \\
		\ynsd{i_{q_i - 2}}{\ell_{q_i-1}}{1} & = \ynsd{i_{q_i - 1}}{\ell_{q_i-1}}{0}
	\end{align*}
	to the DDAE. Then the new DDAE has a maximal matching in its shifting graph.
\end{theorem} 

\begin{proof}
	Let $\mathcal{M}$ denote the maximal matching in the shifting graph before we replace variables and add new equations. Then some equation $F_j$ is matched with the equivalence class corresponding to $\xnsd{i}{\ell}{0}$, where $\ell = \max\{\ell_1,\ldots,\ell_{q_i}\} = \ell_1$. Note that by assumption we have introduced $q_i-1$ new variables yielding $q_i-1$ new equivalence classes in the shift graph. If after replacing we can still match $F_j$ with the equivalence class corresponding to $\xnsd{i}{\ell}{0}$, then we can construct a maximal matching by assigning each new equation to the equivalence class of the variable on its right-hand side and obtain a maximal matching. Each equation is assigned to a highest shift since $\ell_i \ge \ell_{i+1}$ by construction. If after replacing we cannot assign $F_j$ to the equivalence class corresponding to $\xnsd{i}{\ell}{0}$, then at least one of the newly introduced variables must appear in $F_j$ and we can assign $F_j$ to the associated equivalence class of $\ynsd{i_k}{\ell}{0}$ which appears on the right hand side. Thus we can assign all equations 
	\begin{align*}
		\xnsd{i}{\ell}{1} & = \ynsd{i_1}{\ell}{0} \\
		\ynsd{i_1}{\ell}{1} & = \ynsd{i_2}{\ell}{0} \\
		& \vdots \\
		\ynsd{i_{k-1}}{\ell}{1} & = \ynsd{i_{k}}{\ell}{0}
	\end{align*} 
	to the left hand side which are shifted all $\ell$ times by construction of $\ell_i$. The remaining equations can be assigned to the right hand side which are by construction highest shifts.
\end{proof}

\subsection{The Algorithm}
Before we summarize our findings in form of an algorithm, let us briefly recap the results of this section.
\begin{itemize}
	\item We illustrated that we can use the original Pantelides algorithm \cite{Pan88} (see \Cref{alg:PantDAE}) with a different equivalence class to determine which equations need to be shifted. Recall that we may have to differentiate some equations during the shifting step to resolve implicit connections via equivalence classes (cf.\ \cref{subsec:DiffInShift}). To account for the additional differentiation, we present a slight modification of \Cref{alg:PantDAE} in \Cref{alg:PantDAENew}. We emphasize that the notation $\xi_{\coi,\cok}$ for $\xi \in \{\Cof,\cox \}$ has a different meaning for shifting and differentiation. In the case of shifting, the $\cok$ represents the order of shift while in the case of differentiation $\cok$ represents the order of differentiation. This enables us to shift in the shifting graph and differentiate in the differentiation graph with the same notation.
	\item Due to the structure of the equivalence classes, we shift first and then differentiate, see \Cref{prop:PDDAEdiffDoesntEffectShift}. Note that we do not have to resolve implicit connections within the differentiation step, since we only determine the maximal number that each equation needs to be shifted and differentiated. Since a possible implicit connection in the differentiation step results from equations that are already shifted, the implicit connection is resolved automatically by using the same equation with a smaller shift than the maximal shift that was determined in the shifting step.
	\item The differentiation during the shifting step might yield higher-order derivatives. Since the Pantelides algorithm can only handle first-order systems, the resulting \gls{DDAE} must be transformed into one. This can be achieved using trimmed linearization. 
\end{itemize}

The complete methodology is summarized in \Cref{alg:PantDDAE}.

\begin{remark}
	During the shifting and differentiation step, we have to update the graph and the corresponding shifting and differentiation graphs. To unify this procedure, the necessary updates of the graph are summarized in \Cref{UpdateGraph}. This algorithm is called in the modified Pantelides algorithm (\Cref{alg:PantDAENew}).
\end{remark}

\begin{algorithm}
	\caption{Generalization of \Cref{alg:PantDAE} for shifting and differentiation}\label{alg:PantDAENew}	
	\begin{algorithmic}[1]
		\Require \begin{itemize}
			\item[] \hspace{-2em}graph $\Cog = (\Cov_\Coe \dot{\cup} \Cov_\Cov, \Coe)$ 
			\item[] denote by $\shift{\Cof_\coj^{(\ell)}}{\com \tau}$, respectively  $\shift{\cox_\coi^{(\cok)}}{\con \tau} $ the elements of $\Cov_\Coe $, respectively  $\Cov_\Cov$
			\item[] $\code{type}\in \{\code{shift}, \code{dif} \}$
			\item[] graph $\Cog^\code{t} = (\Cov_\Coe^\code{t} \dot{\cup} \Cov_\Cov^\code{t}, \Coe^\code{t})$ of type $\code{type}$	
			\item[]  $\Cov_\Coe^\code{t} = \begin{cases}
			\left\{ \Cof_{\coj, \ell} \mid  \shift{\Cof_\coj^{(\ell)}}{\com \tau} \in \Cov_\Coe \right\} & \text{if } \code{type} == \code{dif} \\
			\left\{ \Cof_{\coj, \ell} \mid \shift{\Cof_\coj^{(\com)}}{\ell \tau} \in \Cov_\Coe \right\} & \text{if } \code{type} == \code{shift} \\
			\end{cases}$							
			\item[]  $\Cov_\Cov^\code{t} = \begin{cases}
			\left\{ \cox_{\coi, \cok} \mid  \shift{\cox_\coi^{(\cok)}}{\con \tau} \in \Cov_\Cov\right\} & \text{if } \code{type} == \code{dif} \\
			\left\{ \cox_{\coi, \cok} \mid \shift{\cox_\coi^{(\con)}}{\cok \tau} \in \Cov_\Cov\right\} & \text{if } \code{type} == \code{shift} \\
			\end{cases}$	
		\end{itemize}
		\Ensure  
		\begin{itemize}
			\item[] \hspace{-2em} graph $\Cog= (\Cov_\Coe \dot{\cup} \Cov_\Cov, \Coe)$
		\end{itemize}
			\Statex
			\State $\code{assign} ( \cox_{\coi, \cok} ) \leftarrow \code{0} $ for each $\cox_{\coi, \cok} \in \Cov_\Cov^\code{t} $
			\For{$\cor=1,\dots,\code{M} $} 
				\State $(\cop,\coq) \leftarrow (\cor,0) $
				\Repeat
					\State $\tilde{\Coe}^\code{t} \gets \Coe^\code{t} \setminus \{ \{\cox_{\coi, \cok},\Cof_{\coj,\ell}\} \in \Coe^\code{t} \mid \cox_{\coi,\cok+ \code{s}} \in \Cov_\Cov^\code{t} \text{ for some } \code{s}>0  \} $ \label{PDAEn5}
					\State $\tilde{\Cov}_\Cov^\code{t} \gets \Cov_\Cov^\code{t} \setminus   \{\cox_{\coi, \cok} \mid \cox_{\coi,\cok+ \code{s}} \in \Cov_\Cov^\code{t} \text{ for some } \code{s}>0  \} $ \label{PDAEn6}
					\State $\code{colorV}(\cox_{\coi, \cok} ) \leftarrow 0 $ for each $\cox_{\coi, \cok} \in \tilde{\Cov}_\Cov^\code{t} $
					\State $\code{colorE}(\Cof_{\coj, \ell} ) \leftarrow 0$ for each $ \Cof_{\coj, \ell} \in \Cov_\Coe^\code{t}  $ \label{PDAEn4}
					\State (\code{pathfound}, \code{assign}, \code{colorV}, \code{colorE}) $\leftarrow$ \text{Algorithm1}($(\Cov_\Coe^\code{t} \dot{\cup} \tilde{\Cov}_\Cov^\code{t},\tilde{\Coe}^\code{t} )$, $\elementVE_{\cop,\coq}$, 
					\Statex \hspace{4em}\code{assign}, \code{colorV}, \code{colorE}) \Comment{apply Augmentpath} 
					\If{\code{pathfound} == \code{false}  } 
						\If{$\code{type}==\code{shift}$}
							\State $\Cog \gets \code{Algorithm3}(\Cog, \Cog^\code{t}, \code{colorV}, \code{colorE}, \code{assign}, \Cof_{\cop,\coq})$\label{PDAEn12}	
							\Statex \Comment{differentiation during the shifting step}
						\EndIf 		
						\State $ \Cog^\code{t} \gets \code{Algorithm6}(\Cog^\code{t}, \code{type}, \code{boolFull} \gets \code{false},\code{colorV}, \code{colorE})$ 
						\Statex \Comment{shift or different the equations}
						\State $ \Cog \gets \code{Algorithm6}(\Cog, \code{type}, \code{boolFull} \gets \code{true}, \code{colorV}, \code{colorE})$	
						\Statex \Comment{shift or different the equations}									
						\ForEach{ $\coi$ with $\code{colorV}(\cox_{\coi, \cok} ) == 1$ and $\Cof_{\coj, \ell} == \code{assign}(\cox_{\coi, \cok} )$}
							\State  $\code{assign}(\cox_{\coi, \cok+1} ) \gets \Cof_{\coj,\ell}$
						\EndFor
					\EndIf  
				\Until{\code{pathfound} = \code{true}  }
			\EndFor	
	\end{algorithmic}
\end{algorithm}

\begin{algorithm}[htb]
	\caption{Pantelides Algorithm for \glspl{DDAE}}\label{alg:PantDDAE}	
	\begin{algorithmic}[1]
		\Require graph $\Cog = (\Cov_\Coe \dot{\cup} \Cov_\Cov, \Coe)$, denote by $\shift{\Cof_\coj^{(\ell)}}{\com \tau}$, respectively  $\shift{\cox_\coi^{(\cok)}}{\con \tau} $ the elements of $\Cov_\Coe $, respectively  $\Cov_\Cov$.
		\Ensure  graph $\Cog$
		\Statex 
		\Statex \textbf{Step 1:} Shifting
		\State $\Cov_\Coe^\code{s} \gets \left\{ \Cof_{\coj, 0} \mid \Cof_\coj \in \Cov_\Coe \right\} $ \label{PDDAE1} 
		\State $\Cov^\code{s}_\Cov \gets \left\{ \cox_{\coi, \cok} \mid \shift{\cox_\coi^{(\con)}}{\cok \tau} \in \Cov_\Cov \right\} $
		\State  $\Coe^\code{s}\gets \{\{\cox_{\coi, \cok}, \Cof_{\coj, \ell}\} \mid \exists \con, \com\in \N_0\text{ s.\,t. } \{\shift{\cox_\coi^{(\con)}}{\cok \tau}, \shift{\Cof_\coj^{(\com)}}{\ell \tau}  \} \in \Coe \}$ \label{PDDAE3}
		\State  $\Cog^\code{s} \gets (\Cov_\Cov^\code{s} \dot{\cup} \Cov^\code{s}_\Coe, \Coe^\code{s}) $ \label{PDDAE4}
		\State  $\Cog\phantom{^\code{s}}$ $\gets$ \code{Algorithm4}($\Cog$, $\code{shift}$, $\Cog^\code{s}$) \label{PDDAE5} \Comment{Generalization of the Pantelides algorithm}
		\Statex
		\Statex \textbf{Step 2:} Trimmed Linearization 
		\State Add equations to graph $\Cog$ according to \Cref{thm:trimmedLinearization}
		\Statex 
		\Statex \textbf{Step 3:} Differentiation
		\State $\Cov^\cod_\Coe \gets\left\{ \Cof_{\coj, \ell} \mid  \shift{\Cof_\coj^{(\ell)}}{\com \tau} \in \tilde{\Cov}_\Coe \right\}  $  \label{PDDAE6}
		\State 	$ \Cov^\cod_\Cov \gets \left\{ \cox_{\coi, \cok} \mid  \shift{\cox_\coi^{(\cok)}}{\con \tau} \in \tilde{\Cov}_\Cov \right\}  $ \label{PDDAE7}
		\State  $\Coe^\cod\gets \{\{\cox_{\coi, \cok}, \Cof_{\coj, \ell}\} \mid \exists \con, \com \text{ s.\,t. } \{\shift{\cox_\coi^{(\cok)}}{\con \tau}, \shift{\Cof_\coj^{(\ell)}}{\com \tau}   \} \in \tilde{\Coe} \}$	
		\State $\Cog^\cod \gets (\Cov^\cod_\Coe \dot{\cup} \Cov^\cod_\Cov, \Coe^\cod)$ \label{PDDAE9}  
		\State $\Cog\phantom{^\cod} \gets \code{Algorithm4}(\Cog, \code{dif}, \Cog^\cod)$ \label{PDDAE10}  \Comment{Generalization of the Pantelides algorithm}
	\end{algorithmic}
\end{algorithm}

\begin{algorithm}[tb]
	\caption{Shift or differentiate equations}\label{UpdateGraph}	
	\begin{algorithmic}[1]
		\Require \begin{itemize}
			\item[] \hspace{-2em}graph $\Cog = (\Cov_\Coe \dot{\cup} \Cov_\Cov, \Coe)$ 
			\item[] $\code{type}\in \{\code{shift}, \code{dif} \}$, $\code{boolFull} \in \{\code{true}, \code{false} \}$
			\item[] $\code{colorV}$, \code{colorE} 
		\end{itemize}
		\Ensure  
		\begin{itemize}
			\item[] \hspace{-2em} graph $\tilde{\Cog}= (\tilde{\Cov}_\Cov \dot{\cup} \tilde{\Cov}_\Coe, \tilde{\Coe})$
		\end{itemize}
			\Statex
			\If{$\code{boolFull} == \code{false}$}
				\State $\cox_{\coi, \cok,0} \gets \cox_{\coi, \cok} \in \Cov_\Cov$ for all $\coi, \cok$
				\State $\Cof_{\coj, \ell, 0} \gets \Cof_{\coj, \ell} \in \Cov_\Coe$ for all $\coj, \ell$
			\ElsIf{$\code{boolFull} == \code{true}$ and $\code{type}==\code{shift}$}
				\State $\cox_{\coi, \cok,\con} \gets \shift{\cox_\coi^{(\con)}}{\cok \tau}  \in \Cov_\Cov$ for all $\coi, \cok, \con $
				\State $\Cof_{\coj, \ell, \com} \gets \shift{\Cof_\coj^{(\com)}}{\ell \tau} \in \Cov_\Coe$ for all $\coj, \ell, \com$
			\Else
				\State $\cox_{\coi, \cok,\con} \gets \shift{\cox_\coi^{(\cok)}}{\con \tau} \in \Cov_\Cov$ for all $\coi, \cok, \con$	
				\State $\Cof_{\coj, \ell, \com} \gets\shift{\Cof_\coj^{(\ell)}}{\com \tau} \in \Cov_\Coe$ for all $\coj, \ell, \com$	
			\EndIf	
			\State $\tilde{\Cov}_\Cov \gets \{\cox_{\coi, \cok,\con}\in \Cov_\Cov \}$
			\State $\tilde{\Cov}_\Coe \gets \{\Cof_{\coj, \ell, \com} \in \Cov_\Coe  \}$
			\State $\tilde{\Coe} \gets \left\{ \{\cox_{\coi, \cok,\con},\Cof_{\coj, \ell, \com}  \} \in \Coe \right\}$
			\If{\code{colorV} == []}
				\State for each $\cox_{\coi,\cok,\con} \in \tilde{\Cov}_\Cov$ define  $\code{colorV}(\cox_{\coi,\cok,\con }) \gets \max \{ \coq \mid  \Cof_{\coj, \ell, \com} \in \tilde{\Cov}_\Coe ,\, \{\cox_{\coi,\cok,\con},  \Cof_{\coj, \ell, \com} \} \in \tilde{\Coe}  \text{ and } \code{colorE}( \Cof_{\coj, \ell, \com}) == \coq \}$
			\EndIf 
			\State $\tilde{\Cov}_\Cov\gets \tilde{\Cov}_\Cov \cup \{\cox_{\coi, \cok+\cop, \con} \mid \cox_{\coi, \cok, \con} \in \tilde{\Cov}_{\Cov},\,   \code{colorV}(\cox_{\coi, \cok, \con})== \coq,\, 1 \le \cop \le \coq  \}$					
			\ForEach{$(\coj,\ell)$ s.\,t.\ $\code{colorE}(\Cof_{\coj, \ell, \com}) == \coq$, $\coq >0$}
				\State $\Cof_{\coj,\ell, \com} \gets \Cof_{\coj, \ell+\coq, \com}$
				\If{$ \code{type} == \code{shift}$}
					\State$\left\{\cox_{\coi, \cok, \con}, \Cof_{\coj, \ell, \com}\right\} \in \tilde{\Coe} \gets \left\{\cox_{\coi, \cok+1, \con}, \Cof_{\coj, \ell, \com}\right\}  $ for each $\coi, \cok, \con$
				\Else
					\State     $\tilde{\Coe} \gets \tilde{\Coe} \cup \left\{\{\cox_{\coi, \cok+\cop, \con},\Cof_{\coj, \ell, \com}\} \mid \cox_{\coi, \cok+\cop, \con} \in \tilde{\Cov}_{\Cov},\,  \{\cox_{\coi, \cok, \con},\Cof_{\coj, \ell,\com} \}\in \tilde{\Coe},\,  1 \le \cop \le \coq \right\}$
				\EndIf
			\EndFor		
			\State $\tilde{\Cog} \gets (\tilde{\Cov}_\Cov \dot{\cup} \tilde{\Cov}_\Coe, \tilde{\Coe})$							
	\end{algorithmic}
\end{algorithm}

\section{Examples}
\label{sec:examples}
We illustrate \Cref{alg:PantDDAE} with two examples. The first two examples, which we discuss in \Cref{subsec:exDDAEComplicated} and \Cref{subsec:exDDAEComplicated2}, are linear toy examples that illustrates the different steps of \Cref{alg:PantDDAE}. The third example (cf.~\Cref{subsec:exHybridNumExperimental}) is a nonlinear DDAE that arises in real-time dynamic substructuring \cite{BurW08}.

\subsection{Academic example 1}
\label{subsec:exDDAEComplicated}
The initial shifting graph of the \gls{DDAE}
\begin{align}
	\label{eq:PDDAEEx}
	\begin{aligned}
		 0 & = x_1 + f_1 \\
		 \dot{x}_1 & = \shift{x_2}{-\tau} + f_2\\
		 \dot{x}_2 & = x_1 + \shift{x_3}{-\tau} + f_3 \\
		 \dot{x}_3 & = x_4 + \shift{x_1}{-\tau} + f_4
		\end{aligned}
\end{align}
is given in \Cref{fig:DAEEx3G1} and in the notation of \Cref{alg:PantDDAE} by the sets 
	\begin{align*}
		\VEs & \defEqual \left\{ \Cof_{1,0}, \Cof_{2,0}, \Cof_{3,0}, \Cof_{4,0}  \right\},\\
		\VVs & \defEqual \left\{ \cox_{1,0}, \cox_{2,0}, \cox_{3,0},\cox_{4,0}, \cox_{1,-1}, \cox_{2,-1}, \cox_{3,-1} \right\}, \\
		E^\mathrm{s} & \defEqual \left\{ \{\Cof_{1,0}, \cox_{1,0} \},\{\Cof_{2,0}, \cox_{1,0} \}, \{\Cof_{2,0}, \cox_{2,-1} \}, \{\Cof_{3,0}, \cox_{1,0} \},\{\Cof_{3,0}, \cox_{2,0} \}, \right. \\
		& \hspace{4ex}\left. \{\Cof_{3,0}, \cox_{3,-1} \},\{\Cof_{4,0}, \cox_{3,0} \},\{\Cof_{4,0}, \cox_{4,0} \},\{\Cof_{4,0}, \cox_{1,-1} \}  \right\},
	\end{align*} 
	where each $\cox_{i,j}$ denotes one equivalence class. We observe that there is no maximal matching in the shifting graph and we may choose either $F_1$ or $F_2$ as exposed (cf.~\Cref{fig:DAEEx3G2}). 
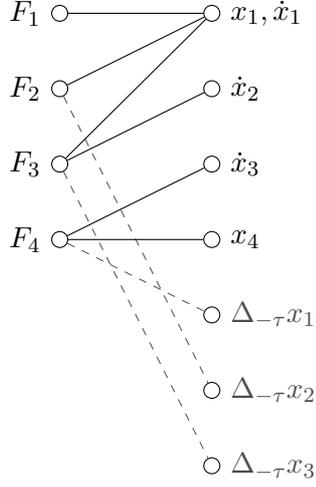
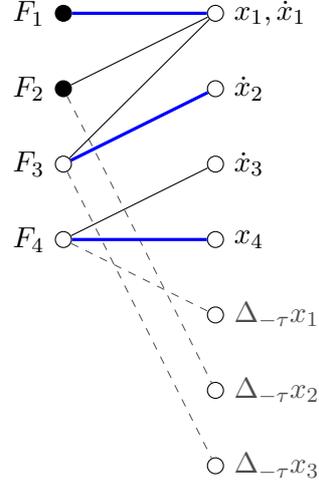
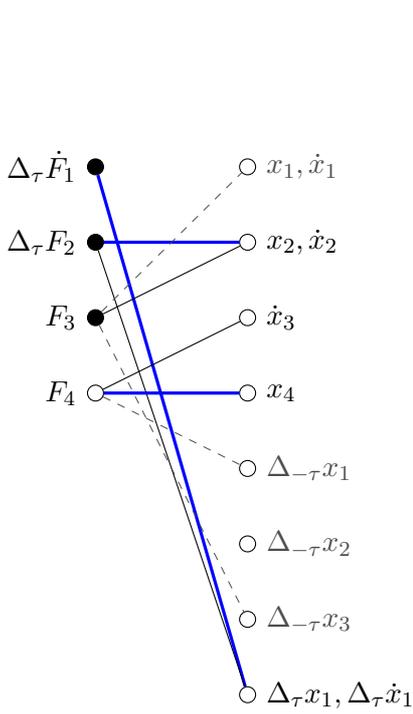
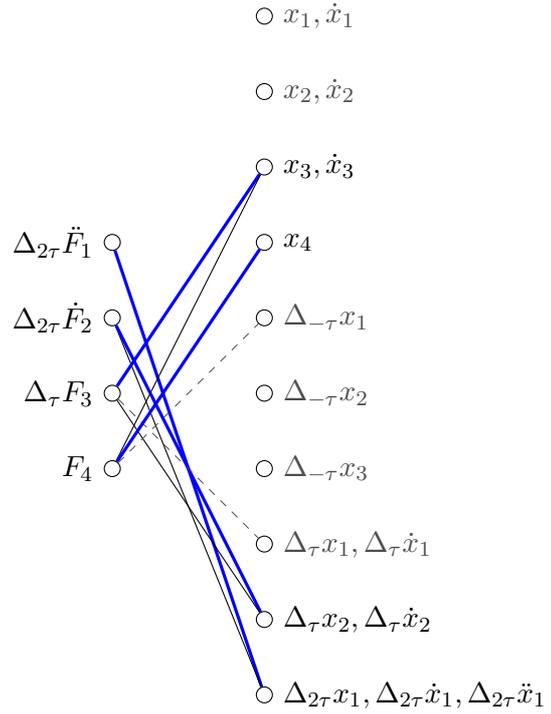
\begin{figure}
	\begin{subfigure}[t]{0.45\textwidth}
		\centering
		\begin{tikzpicture} 
			\enode{1}{$F_1$}
			\enode{2}{$F_2$}
			\enode{3}{$F_3$}
			\enode{4}{$F_4$}
			\vnode{1}{$x_1, \dot{x}_1$}{}
			\vnode{2}{$\dot{x}_2$}{}
			\vnode{3}{$\dot{x}_3$}{}
			\vnode{4}{$x_4$}{}
			\vnode[low]{5}{$\shift{x_1}{-\tau}$}{}
			\vnode[low]{6}{$\shift{x_2}{-\tau}$}{}
			\vnode[low]{7}{$\shift{x_3}{-\tau}$}{}
			\edge{e1}{v1}
			\edge{e2}{v1}
			\edge[low]{e2}{v6}
			\edge{e3}{v1}
			\edge{e3}{v2}
			\edge[low]{e3}{v7}
			\edge{e4}{v3}
			\edge{e4}{v4}
			\edge[low]{e4}{v5}						
		\end{tikzpicture}	
		\caption{Initial shifting graph}
		\label{fig:DAEEx3G1}
	\end{subfigure}
	\qquad
	\begin{subfigure}[t]{0.45\textwidth}
		\centering
		\begin{tikzpicture} 
			\enode[try]{1}{$F_1$}
			\enode[try]{2}{$F_2$}
			\enode{3}{$F_3$}
			\enode{4}{$F_4$}
			\vnode{1}{$x_1, \dot{x}_1$}{}
			\vnode{2}{$\dot{x}_2$}{}
			\vnode{3}{$\dot{x}_3$}{}
			\vnode{4}{$x_4$}{}
			\vnode[low]{5}{$\shift{x_1}{-\tau}$}{}
			\vnode[low]{6}{$\shift{x_2}{-\tau}$}{}
			\vnode[low]{7}{$\shift{x_3}{-\tau}$}{}
			\edge[assign]{e1}{v1}
			\edge{e2}{v1}
			\edge[low]{e2}{v6}
			\edge{e3}{v1}
			\edge[assign]{e3}{v2}
			\edge[low]{e3}{v7}
			\edge{e4}{v3}
			\edge[assign]{e4}{v4}
			\edge[low]{e4}{v5}		
		\end{tikzpicture}	
		\caption{Shifting graph with assignment}
		\label{fig:DAEEx3G2}
	\end{subfigure}\\[1em]
	\begin{subfigure}[t]{0.45\textwidth}
		\centering
		\begin{tikzpicture} 
			\enode[try]{1}{$\shift{\dot{F}_1}{\tau}$}
			\enode[try]{2}{$\shift{F_2}{\tau}$}
			\enode[try]{3}{$F_3$}
			\enode{4}{$F_4$}
			\vnode[low]{1}{$x_1, \dot{x}_1$}{}
			\vnode{2}{$x_2, \dot{x}_2$}{}
			\vnode{3}{$\dot{x}_3$}{}
			\vnode{4}{$x_4$}{}
			\vnode[low]{5}{$\shift{x_1}{-\tau}$}{}
			\vnode[low]{6}{$\shift{x_2}{-\tau}$}{}
			\vnode[low]{7}{$\shift{x_3}{-\tau}$}{}
			\vnode{8}{$\shift{x_1}{\tau}, \shift{\dot{x}_1}{\tau}$}{}
			\edge[assign]{e1}{v8}
			\edge{e2}{v8}
			\edge[assign]{e2}{v2}
			\edge[low]{e3}{v1}
			\edge{e3}{v2}
			\edge[low]{e3}{v7}
			\edge{e4}{v3}
			\edge[assign]{e4}{v4}
			\edge[low]{e4}{v5}		
		\end{tikzpicture}	
		\caption{Shifting graph after differentiating $F_1$ and shifting $\dot{F}_1$ and $F_2$ with assignment}
		\label{fig:DAEEx3G3}	
	\end{subfigure}
	\qquad
	\begin{subfigure}[t]{0.45\textwidth}
	\centering
	\begin{tikzpicture} 
		\enode{4}{$\shift{\ddot{F}_1}{2 \tau}$}
		\enode{5}{$\shift{\dot{F}_2}{2 \tau}$}
		\enode{6}{$\shift{F_3}{\tau}$}
		\enode{7}{$F_4$}
		\vnode[low]{1}{$x_1, \dot{x}_1$}{}
		\vnode[low]{2}{$x_2, \dot{x}_2$}{}
		\vnode{3}{$x_3, \dot{x}_3$}{}
		\vnode{4}{$x_4$}{}
		\vnode[low]{5}{$\shift{x_1}{-\tau}$}{}
		\vnode[low]{6}{$\shift{x_2}{-\tau}$}{}
		\vnode[low]{7}{$\shift{x_3}{-\tau}$}{}
		\vnode[low]{8}{$\shift{x_1}{\tau}, \shift{\dot{x}_1}{\tau}$}{}
		\vnode{9}{$\shift{x_2}{\tau}, \shift{\dot{x}_2}{\tau}$}{}
		\vnode{10}{$\shift{x_1}{2\tau}, \shift{\dot{x}_1}{2\tau}, \shift{\ddot{x}_1}{2\tau}$}{}
		\edge[assign]{e4}{v10}
		\edge{e5}{v10}
		\edge[assign]{e5}{v9}
		\edge[low]{e6}{v8}
		\edge{e6}{v9}
		\edge[assign]{e6}{v3}
		\edge{e7}{v3}
		\edge[assign]{e7}{v4}
		\edge[low]{e7}{v5}								
	\end{tikzpicture}	
	\caption{Shifting graph after differentiating $\shift{\dot{F}_1}{\tau}$, $\shift{F_2}{\tau}$ and shifting $\shift{\ddot{F}_1}{\tau}$, $\shift{\dot{F}_2}{\tau}$ and $F_3$  }
	\label{fig:DAEEx3G4}	
\end{subfigure}
	\caption{Visualization of the shifting step from \Cref{alg:PantDDAE} for the \gls{DDAE}~\eqref{eq:PDDAEEx}}
\end{figure}
Since both are connected via a matching, we shift $F_1$ and $F_2$. Note, that $F_3$ has a path to $F_1$ and $F_2$ but the path is not an alternating path. Thus, we do not shift $F_3$. The equations $F_1$ and $F_2$ are only implicitly connected via the equivalence class $\{x_1, \dot{x}_1\}$. Constructing the linear integer program as described in \Cref{subsec:DiffInShift} yields 
	\begin{align*}
	\min \nu_1 + \nu_2 \quad \text{ such that }\quad \begin{bmatrix}
	1  & -1 
	\end{bmatrix} \begin{bmatrix}
	\nu_1 \\ \nu_2 
	\end{bmatrix} = \begin{bmatrix}
	1
	\end{bmatrix}, \quad \nu_1, \nu_2 \in \N_0
	\end{align*}
	with the solution $\nu_1 = 1$, $\nu_2 = 0$. Consequently, we differentiate $F_1$. In the resulting graph \Cref{fig:DAEEx3G3} a possible matching is giving by $\mathcal{M} = \{ \{\Cof_{1,1}, \cox_{1,1} \}, \{\Cof_{2,1}, \cox_{2,0} \}, \{\Cof_{4,0}, \cox_{4,0} \} \}$. The matching is not maximal and the exposed vertex $\Cof_{3,1}$ is connected through an alternating path to $\Cof_{1,1}$ and $\Cof_{2,1}$. Thus we need to shift these three equations. The linear program for these three equations
	\begin{align*}
	\min \nu_1 + \nu_2+ \nu_3 \quad \text{ such that } \quad \begin{bmatrix}
	0 & 1  & -1 \\ 1 & -1 & 0
	\end{bmatrix} \begin{bmatrix}
	\nu_1 \\ \nu_2 \\ \nu_3
	\end{bmatrix} = \begin{bmatrix}
	1 \\0
	\end{bmatrix}, \quad \nu_1, \nu_2, \nu_3 \in \N_0
	\end{align*}	
	with solution $\nu_1 = \nu_2 = 1$ and $\nu_3 = 0$ yields that we differentiate the first two equations. The resulting shifting graph \Cref{fig:DAEEx3G4} provides a maximal matching. 
	The differentiation during the shifting step results in a single variable which is differentiated twice or more, namely $\shift{\ddot{x}_1}{2 \tau}$. Since the Pantelides algorithm can only handle derivatives up to first-order, we replace this variable in the trimmed linearization step by $\dot{y}_{1_0}$ and add the equation $G_{1_0}$ given by 
	\begin{align*}
		y_{1_0} = \dot{x}_1. 
	\end{align*} 
	Since $\shift{\ddot{x}_1}{2\tau}$ appears, we shift $G_{1_0}$ twice. To enlarge the full graph by this variable and equation, we to rename $y_{1_0}$ to $x_5$ and $G_{1_0}$ to $F_5$. This yields the graph of the \gls{DDAE}
	\begin{align*}
		\tilde{V}_\mathrm{E} & \defEqual \left\{ \codeFnsd{1}{2}{2}, \codeFnsd{2}{2}{1}, \codeFnsd{3}{1}{0}, \codeFnsd{4}{0}{0}, \codeFnsd{5}{2}{0}, \right\}, \\
		\tilde{V}_\mathrm{V} & \defEqual  \left\{ \codexnsd{1}{1}{0}, \codexnsd{1}{2}{0}, \codexnsd{1}{2}{1}, \codexnsd{2}{1}{0}, \codexnsd{2}{1}{1}, \codexnsd{3}{0}{0}, \codexnsd{3}{0}{1}, \codexnsd{4}{0}{0}, \codexnsd{5}{2}{0}, \codexnsd{5}{2}{1} \right\}, \\
		\tilde{E} & \defEqual \left\{ \{\codeFnsd{1}{2}{2}, \codexnsd{1}{2}{0} \}, \{\codeFnsd{1}{2}{2}, \codexnsd{5}{2}{0} \}, \{\codeFnsd{1}{2}{2}, \codexnsd{5}{2}{1} \}, \{\codeFnsd{2}{2}{1}, \codexnsd{2}{1}{0} \}, \{\codeFnsd{2}{2}{1}, \codexnsd{2}{1}{1} \}, \right. \\
		& \hspace{4.2ex}\left.  \{\codeFnsd{2}{2}{1}, \codexnsd{5}{2}{0} \}, \{\codeFnsd{2}{2}{1}, \codexnsd{5}{2}{1} \}, \{\codeFnsd{3}{1}{0}, \codexnsd{1}{1}{0}\}, \{\codeFnsd{3}{1}{0}, \codexnsd{2}{1}{1}\}, \{\codeFnsd{3}{1}{0}, \codexnsd{3}{0}{0}\},  \{\codeFnsd{4}{0}{0}, \codexnsd{3}{0}{1}\},  \right. \\
		& \hspace{4.2ex}\left. \{\codeFnsd{4}{0}{0}, \codexnsd{4}{0}{0}\},  \{\codeFnsd{5}{2}{0}, \codexnsd{1}{2}{1}\},  \{\codeFnsd{5}{2}{0}, \codexnsd{5}{2}{0}\}           \right\}. 
	\end{align*} 
	Now we construct the differentiation graph for the equations $\shift{\ddot{F}_1}{2 \tau}, \shift{\dot{F}_2}{2 \tau}, \shift{F_3}{\tau}, F_4$ and $\Fnsd{5}{2}{0}$ given by   
	\begin{align*}
		\VEd & \defEqual \left\{ \Cof_{1,2}, \Cof_{2,1}, \Cof_{3,0}, \Cof_{4,0},  \Cof_{5,0}   \right\},\\
		\VVd & \defEqual \left\{ \cox_{1,0}, \cox_{2,0}, \cox_{3,0},\cox_{4,0}, \cox_{1,1}, \cox_{2,1}, \cox_{3,1}, \cox_{5,0}, \cox_{5,1} \right\}, \\
		E^\mathrm{d} & \defEqual \left\{ \{\Cof_{1,2}, \cox_{1,0} \},\{\Cof_{1,2}, \cox_{1,1} \}, \{\Cof_{1,2}, \cox_{5,1} \}, \{\Cof_{2,1}, \cox_{2,0} \},\{\Cof_{2,1}, \cox_{1,1} \},\{\Cof_{2,1}, \cox_{2,1} \},\{\Cof_{2,1}, \cox_{5,1} \}, \right. \\
		& \hspace{4ex}\left. \{\Cof_{3,0}, \cox_{1,0} \},\{\Cof_{3,0}, \cox_{3,0} \},\{\Cof_{3,0}, \cox_{2,1} \}, \{\Cof_{4,0}, \cox_{4,0} \},\{\Cof_{4,0}, \cox_{3,1} \},\{\Cof_{5,0}, \cox_{1,1} \}, \{\Cof_{5,0}, \cox_{5,0} \}    \right\},
	\end{align*} 
	or respectively \Cref{fig:DAEEx3G5}. Note that the variable $\shift{x_1}{-\tau}$ does not occur in the differentiation graph since it is shifted. The vertex $\shift{F_3}{\tau}$ is exposed and connected through an alternating path to $\shift{\ddot{F}_1}{2 \tau}$ and $\shift{\dot{F}_2}{2 \tau}$. Furthermore $\Fnsd{5}{2}{0}$ is exposed. Thus we differentiate all four equations which results in \Cref{fig:DAEEx3G6}. The existence of a maximal matching in \Cref{fig:DAEEx3G5} leads to the termination of \Cref{alg:PantDDAE}. We conclude that we shift the first, second and fifth equation twice, the third equation once and that we differentiate the first equation three time, the second equation twice and the third and fifth equation once.

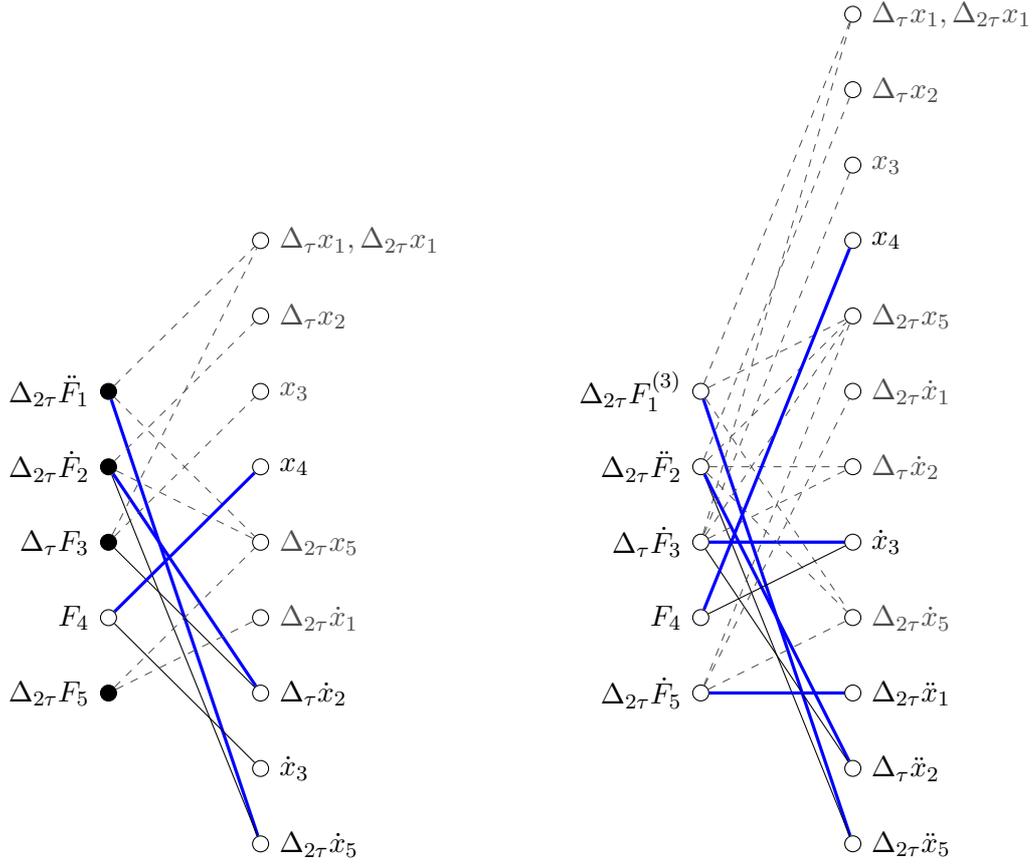
\begin{figure}[t]
	\begin{subfigure}[t]{0.45\textwidth}
		\centering
		\begin{tikzpicture} 
			\enode[try]{3}{$\shift{\ddot{F}_1}{2 \tau}$}
			\enode[try]{4}{$\shift{\dot{F}_2}{2 \tau}$}
			\enode[try]{5}{$\shift{F_3}{\tau}$}
			\enode{6}{$F_4$}
			\enode[try]{7}{$\Fnsd{5}{2}{0}$}
			\vnode[low]{1}{$\shift{x_1}{\tau},\shift{x_1}{2\tau} $}{}
			\vnode[low]{2}{$\shift{x_2}{\tau}$}{}
			\vnode[low]{3}{$x_3$}{}
			\vnode{4}{$x_4$}{}
			\vnode[low]{5}{$\xnsd{5}{2}{0}$}{}	
			\vnode[low]{6}{$\shift{\dot{x}_1}{2\tau} $}{}
			\vnode{7}{$\shift{\dot{x}_2}{\tau}$}{}
			\vnode{8}{$\dot{x}_3$}{}
			\vnode{9}{$\shift{\dot{x}_5}{2\tau} $}{}
			\edge[low]{e3}{v1}
			\edge[low]{e3}{v5}
			\edge[assign]{e3}{v9}
			\edge[low]{e4}{v5}
			\edge{e4}{v9}
			\edge[low]{e4}{v2}
			\edge[assign]{e4}{v7}
			\edge{e5}{v7}
			\edge[low]{e5}{v1}	
			\edge[low]{e5}{v3}
			\edge[assign]{e6}{v4}
			\edge{e6}{v8}
			\edge[low]{e7}{v5}
			\edge[low]{e7}{v6}										
		\end{tikzpicture}	
		\caption{Initial differentiation graph with assignment}
		\label{fig:DAEEx3G5}
	\end{subfigure}
	\qquad
	\begin{subfigure}[t]{0.45\textwidth}
		\centering
		\begin{tikzpicture} 
			\enode{6}{$\shift{F_1^{(3)}}{2 \tau}$}
			\enode{7}{$\shift{\ddot{F}_2}{2 \tau}$}
			\enode{8}{$\shift{\dot{F}_3}{\tau}$}
			\enode{9}{$F_4$}
			\enode{10}{$\Fnsd{5}{2}{1}$}
			\vnode[low]{1}{$\shift{x_1}{\tau},\shift{x_1}{2\tau} $}{}
			\vnode[low]{2}{$\shift{x_2}{\tau}$}{}
			\vnode[low]{3}{$x_3$}{}
			\vnode{4}{$x_4$}{}
			\vnode[low]{5}{$\xnsd{5}{2}{0}$}{}
			\vnode[low]{6}{$\shift{\dot{x}_1}{2\tau} $}{}
			\vnode[low]{7}{$\shift{\dot{x}_2}{\tau}$}{}
			\vnode{8}{$\dot{x}_3$}{}
			\vnode[low]{9}{$\shift{\dot{x}_5}{2\tau} $}{}
			\vnode{10}{$\shift{\ddot{x}_1}{2 \tau}$}{}
			\vnode{11}{$\shift{\ddot{x}_2}{\tau}$}{}
			\vnode{12}{$\shift{\ddot{x}_5}{2\tau} $}{}	
			\edge[low]{e6}{v1}
			\edge[low]{e6}{v5}
			\edge[low]{e6}{v9}
			\edge[assign]{e6}{v12}
			\edge[low]{e7}{v5}
			\edge[low]{e7}{v9}
			\edge[low]{e7}{v2}
			\edge[low]{e7}{v7}
			\edge{e7}{v12}
			\edge[assign]{e7}{v11}
			\edge[low]{e8}{v7}
			\edge[low]{e8}{v1}	
			\edge[low]{e8}{v3}
			\edge{e8}{v11}
			\edge[low]{e8}{v5}	
			\edge[assign]{e8}{v8}	
			\edge[assign]{e9}{v4}
			\edge{e9}{v8}	
			\edge[low]{e10}{v5}
			\edge[low]{e10}{v6}			
			\edge[low]{e10}{v9}
			\edge[assign]{e10}{v10}		
		\end{tikzpicture}	
		\caption{Differentiation graph after differentiating $\shift{\ddot{F}_1}{2\tau}$, $\shift{\dot{F}_2}{2\tau}$ and $\shift{F_3}{\tau}$ }
		\label{fig:DAEEx3G6}	
	\end{subfigure}
	\caption{Visualization of the differentiation step from \Cref{alg:PantDDAE} for the \gls{DDAE}~\eqref{eq:PDDAEEx}}
\end{figure}

\subsection{Academic example 2}
\label{subsec:exDDAEComplicated2}
In this example we finish the already partly executed Pantelides algorithm for DDAEs for  \eqref{eq:LPnotSolvable}. We have already seen in in \Cref{ex:LPNotSolvableContinued} that the shifting step results in equations $\shift{F_1}{\tau}$, $\shift{\dot{F}_2}{\tau}$ and $\shift{\dot{F}_3}{\tau}$. We detect in the trimmed linearization step that the only variable, depending on second or higher-order derivative is given by $\xnsd{1}{1}{2}$. Thus $\xnsd{1}{1}{2}$ and $\xnsd{1}{1}{1}$ are replaced by $\xnsd{4}{1}{1}$ and $\xnsd{4}{1}{0}$, respectively. To add the equation $\xnsd{4}{1}{0} = \xnsd{1}{1}{1}$ to the graph, the vertex $\xnsd{1}{1}{1}$ is added to the graph. This results in the differentiation graph \Cref{fig:DAEEx4G1} where $\shift{F_1}{\tau}$ is exposed and thus differentiated. The resulting graph \Cref{fig:DAEEx4G2} has a maximal matching and thus \Cref{alg:PantDDAE} terminates. 
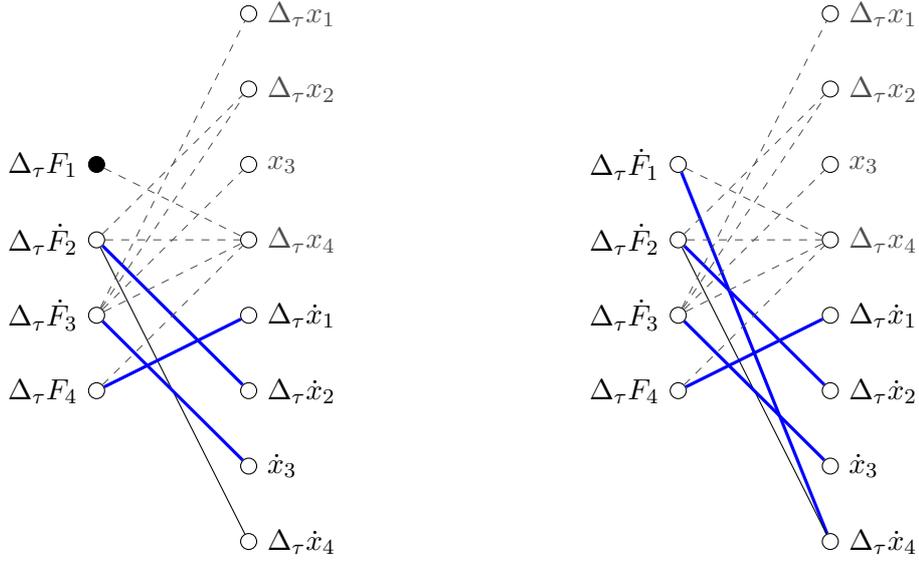
\begin{figure}[t]
	\begin{subfigure}[t]{0.45\textwidth}
		\centering
		\begin{tikzpicture} 
			\enode[try]{3}{$\shift{F_1}{\tau}$}
			\enode{4}{$\shift{\dot{F}_2}{\tau}$}
			\enode{5}{$\shift{\dot{F}_3}{\tau}$}
			\enode{6}{$\shift{F_4}{\tau}$}
			\vnode[low]{1}{$\shift{x_1}{\tau}$}{}
			\vnode[low]{4}{$\shift{x_4}{\tau}$}{}
			\vnode{8}{$\shift{\dot{x}_4}{\tau}$}{}
			\vnode[low]{2}{$\shift{x_2}{\tau}$}{}
			\vnode{6}{$\shift{\dot{x}_2}{\tau}$}{}
			\vnode[low]{3}{$x_3 $}{}
			\vnode{7}{$\dot{x}_3$}{}
			\vnode{5}{$\xnsd{1}{1}{1}$}{}
			\edge[low]{e3}{v4}
			\edge[low]{e4}{v4}
			\edge{e4}{v8}
			\edge[low]{e4}{v2}
			\edge[assign]{e4}{v6}
			\edge[low]{e5}{v1}
			\edge[low]{e5}{v4}
			\edge[low]{e5}{v2}	
			\edge[low]{e5}{v3}
			\edge[assign]{e5}{v7}	
			\edge[low]{e6}{v4}
			\edge[assign]{e6}{v5}								
		\end{tikzpicture}	
		\caption{Initial differentiation graph with assignment}
		\label{fig:DAEEx4G1}
	\end{subfigure}
	\qquad
	\begin{subfigure}[t]{0.45\textwidth}
		\centering
		\begin{tikzpicture} 
			\enode{3}{$\shift{\dot{F}_1}{\tau}$}
			\enode{4}{$\shift{\dot{F}_2}{\tau}$}
			\enode{5}{$\shift{\dot{F}_3}{\tau}$}
			\enode{6}{$\shift{F_4}{\tau}$}
			\vnode[low]{1}{$\shift{x_1}{\tau}$}{}
			\vnode[low]{4}{$\shift{x_4}{\tau}$}{}
			\vnode{8}{$\shift{\dot{x}_4}{\tau}$}{}
			\vnode[low]{2}{$\shift{x_2}{\tau}$}{}
			\vnode{6}{$\shift{\dot{x}_2}{\tau}$}{}
			\vnode[low]{3}{$x_3 $}{}
			\vnode{7}{$\dot{x}_3$}{}
			\vnode{5}{$\xnsd{1}{1}{1}$}{}
			\edge[low]{e3}{v4}
			\edge[assign]{e3}{v8}
			\edge[low]{e4}{v4}
			\edge{e4}{v8}
			\edge[low]{e4}{v2}
			\edge[assign]{e4}{v6}
			\edge[low]{e5}{v1}
			\edge[low]{e5}{v4}
			\edge[low]{e5}{v2}	
			\edge[low]{e5}{v3}
			\edge[assign]{e5}{v7}	
			\edge[low]{e6}{v4}
			\edge[assign]{e6}{v5}		
		\end{tikzpicture}	
		\caption{Differentiation graph after differentiating $\shift{F_1}{\tau}$ with assignments}
		\label{fig:DAEEx4G2}	
	\end{subfigure}
	\caption{Visualization of the differentiation step from \Cref{alg:PantDDAE} for the \gls{DDAE}~\eqref{eq:LPnotSolvable}}
\end{figure}

\subsection{Hybrid numerical-experimental model}
\label{subsec:exHybridNumExperimental}

In earthquake engineering, it is common to use a hybrid numerical-experimental approach to reduce unmanageable costs in testing complex dynamical systems \cite{BurW08}. The complete model is subdivided in a numerical part and an experimental component, which interact in real-time. A state delay arises naturally by transferring numerical results via hydraulic actuators to the experiment \cite{HorIKN99}. This delay essentially implies that the experiment is delayed compared to the numerical simulation. Here, we discuss a coupled oscillator-pendulum system, which is analyzed in \cite{KryBGHW06,Ung20} and depicted in
 \Cref{fig:substructuringPMSD}.
\begin{figure}
	\centering
	\begin{subfigure}[b]{0.4\textwidth}
		\centering
		\begin{tikzpicture}
		\tikzstyle{spring}=[thick,decorate,decoration={zigzag,pre length=0.3cm,post length=0.3cm,segment length=6}]
		\tikzstyle{damper}=[thick,decoration={markings,  
			mark connection node=dmp,
			mark=at position 0.5 with 
			{
				\node (dmp) [thick,inner sep=0pt,transform shape,rotate=-90,minimum width=15pt,minimum height=3pt,draw=none] {};
				\draw [thick] ($(dmp.north east)+(2pt,0)$) -- (dmp.south east) -- (dmp.south west) -- ($(dmp.north west)+(2pt,0)$);
				\draw [thick] ($(dmp.north)+(0,-5pt)$) -- ($(dmp.north)+(0,5pt)$);
			}
		}, decorate]
		
		\node (Mass) [fill,black!30,minimum width=3cm,minimum height=1.5cm] at (0,0) {};
		\node [left,xshift=-.2cm] at (Mass) {$m_1$};
		\node [above right] at (Mass) {$(\mathbf{x}_1,\mathbf{y}_1)$};
		\node (ground) at (Mass.south) [yshift=-2.3cm,fill,pattern=north east lines,draw=none,minimum width=5cm,minimum height=.3cm,anchor=north] {};
		\draw [very thick] (ground.north west) -- (ground.north east);
		\draw [spring] ($(ground.north) - (1,0)$) -- ($(Mass.south)-(1,0)$) node [midway, left,xshift=-.2cm] {$k$};
		\draw [damper] ($(ground.north) + (1,0)$) -- ($(Mass.south)+(1,0)$) node [midway, left,xshift=-.4cm] {$c$};
		
		\draw[gray,->] ($(Mass) + (0,0)$) -- ($(Mass) + (2.5,0)$) node[black,right] {$\mathbf{x}$};
		\draw[gray,->] ($(Mass) + (0,0)$) -- ($(Mass) + (0,1.5)$) node[black,right] {$\mathbf{y}$};
		
		\node (pendulum) at ($(Mass) + (2.5,-1.5)$) {};
		
		\draw[gray,dashed,thick,->] (pendulum) arc (-31:-6:2.9cm);
		\draw[gray,dashed,thick,->] (pendulum) arc (-31:-56:2.9cm);
		
		\draw [fill=black] (pendulum) circle (0.15);
		\node [right,xshift=.2cm] at (pendulum) {$m_2$};
		\node [below,yshift=-.15cm] at (pendulum) {$(\mathbf{x}_2,\mathbf{y}_2)$};
		\draw [very thick] (0,0) -- (pendulum) node [pos=0.7,above right] {$l$};			
		
		\draw[fill=black] (Mass) circle (0.1);
		\end{tikzpicture}
		\caption{Fully coupled system}
		\label{fig:substructuringPMSD:coupled}
	\end{subfigure}
	\begin{subfigure}[b]{0.58\textwidth}
		\centering
		\begin{tikzpicture}
		\tikzstyle{spring}=[thick,decorate,decoration={zigzag,pre length=0.3cm,post length=0.3cm,segment length=6}]
		\tikzstyle{damper}=[thick,decoration={markings,  
			mark connection node=dmp,
			mark=at position 0.5 with 
			{
				\node (dmp) [thick,inner sep=0pt,transform shape,rotate=-90,minimum width=15pt,minimum height=3pt,draw=none] {};
				\draw [thick] ($(dmp.north east)+(2pt,0)$) -- (dmp.south east) -- (dmp.south west) -- ($(dmp.north west)+(2pt,0)$);
				\draw [thick] ($(dmp.north)+(0,-5pt)$) -- ($(dmp.north)+(0,5pt)$);
			}
		}, decorate]
		
		\node (Mass) [fill,black!30,minimum width=2cm,minimum height=1cm] at (0,0) {};
		\node (Fext) at ($(Mass) + (0,1.2)$) {};
		\draw [very thick,->] ($(Mass)$) -- (Fext) node[below right] {$F_{\mathrm{ext}}$};
		\node [left,xshift=-.2cm] at (Mass) {$m_1$};
		\node (ground) at (Mass.south) [yshift=-1.5cm,fill,pattern=north east lines,draw=none,minimum width=2.5cm,minimum height=.3cm,anchor=north] {};
		\draw [very thick] (ground.north west) -- (ground.north east);
		\draw [spring] ($(ground.north) - (.4,0)$) -- ($(Mass.south)-(.4,0)$) node [midway, left,xshift=-.15cm] {$k$};
		\draw [damper] ($(ground.north) + (.4,0)$) -- ($(Mass.south)+(.4,0)$) node [midway, right,xshift=.3cm] {$c$};
		
		\draw[fill=black] (Mass) circle (0.1);
		
		\node (numMod) [align=center,anchor=north] at ($(ground.south)-(0,.2)$) {\textsc{numerical}\\\textsc{model}};
		
		\draw[rounded corners] ($(ground.west |- Fext) + (-.2,.2)$) rectangle ($(ground.east |- numMod.south) + (.2,-.2)$);
		
		\begin{scope}[xshift=5cm]
		\node (Mass2) {};
		\node (Fmeasured) at ($(Mass2) + (0,1.2)$) {};
		\draw [very thick,->] ($(Mass2)$) -- (Fmeasured) node[below right] {$F_{\mathrm{pendulum}}$};
		\draw[fill=black] (Mass2) circle (0.1);
		\node (ground2) at ($(Mass2)-(0,0.5)$) [xshift=.7cm,yshift=-1.5cm,fill,pattern=north east lines,draw=none,minimum width=2.5cm,minimum height=.3cm,anchor=north] {};
		\draw [very thick] (ground2.north west) -- (ground2.north east);
		
		\node (pendulum) at ($(Mass2) + (1.1,-1.5)$) {};
		
		\draw[gray,dashed,thick,->] (pendulum) arc (-53.7:-33.7:1.86cm);
		\draw[gray,dashed,thick,->] (pendulum) arc (-53.7:-73.7:1.86cm);
		
		\draw [fill=black] (pendulum) circle (0.15);
		\node [right,xshift=.2cm] at (pendulum) {$m_2$};
		\draw [very thick] (0,0) -- (pendulum) node [pos=0.7,above right] {$l$};			
		
		\node (actuator) at ($(Mass2) - (0,0.9)$) {};
		\node (actBottom) at (actuator |- ground2.north) {};
		\draw [fill=black!50] ($(actBottom) - (0.1,0)$) rectangle ($(actuator) + (0.1,0)$);
		\draw ($(actuator) - (0.05,0)$) -- ($(actuator) - (0.05,-.55)$) -- ($(actuator) - (0.4,-.55)$) -- ($(actuator |- Mass2) - (0.4,0.1)$) -- ($(actuator |- Mass2) + (0.4,-0.1)$) -- ($(actuator) + (0.4,.55)$) -- ($(actuator) + (0.05,.55)$) -- ($(actuator) + (0.05,0)$);
		
		\node (expAct) [align=center,anchor=north] at ($(ground2.south)-(0,.2)$) {\textsc{experiment}\\\textsc{+ actuator}};
		
		\draw[rounded corners] ($(ground2.west |- Fmeasured) + (-.2,.2)$) rectangle ($(ground2.east |- expAct.south) + (.2,-.2)$);
		\end{scope}
		
		\draw [thick,->,shorten >=0.5cm,shorten <=0.5cm,] (Mass -| ground.east) -- (Mass2 -| ground2.west) node[midway,align=center] {\small adjust\\[.2em]\small position};
		
		\draw [thick,->,shorten >=0.5cm,shorten <=0.5cm,] (ground2.west) -- (ground.east) node[midway,align=center] {\small send\\[.2em]\small $F_{\mathrm{pendulum}}$};
		\end{tikzpicture}
		\caption{Hybrid numerical-experimental setup}
		\label{fig:substructuringPMSD:hybrid}
	\end{subfigure}
	\caption{Real-time dynamic substructuring for a coupled oscillator-pendulum system}
	\label{fig:substructuringPMSD}
\end{figure}
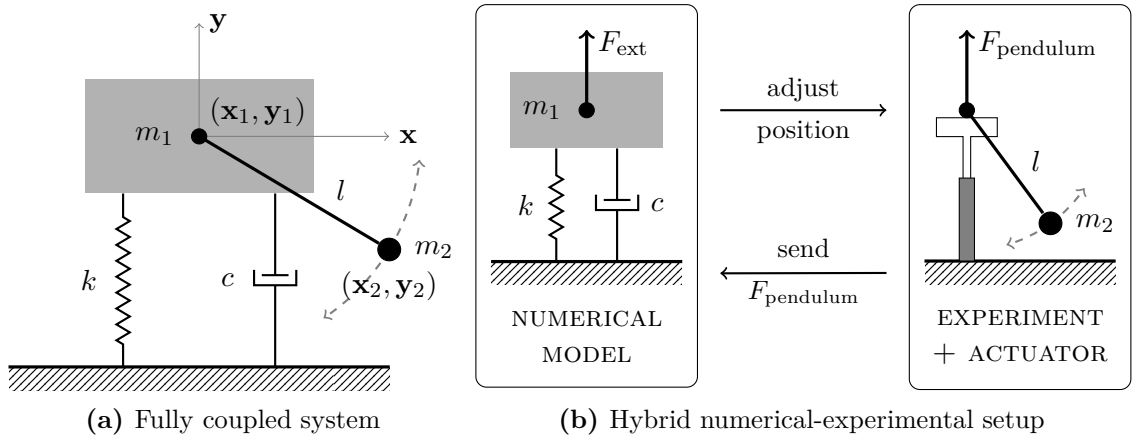
Following \cite{Ung20} the model equations are given by
\begin{subequations}
	\label{eq:PMSD:hybrid}
	\begin{align*}
		m_1\ddot{\mathbf{y}}_1 + c\dot{\mathbf{y}}_1+ k\mathbf{y}_1 &= -2(\shift{\lambda}{-\tau}) ( \shift{\mathbf{y}_2}{-\tau}-\shift{\mathbf{y}_1}{-\tau}) - m_2\mathsf{g},\\
		m_2 \shift{\ddot{x}_2}{-\tau} &= -2(\shift{\lambda}{-\tau}) \shift{x_2}{-\tau},\\
		m_2 \shift{\ddot{\mathbf{y}}_2}{-\tau} &= -2 (\shift{\lambda}{-\tau}) (\shift{\mathbf{y}_2}{-\tau}-\shift{\mathbf{y}_1}{-\tau})- m_2 \mathsf{g},\\
		0 &= \shift{x_2}{-\tau}^2 + (\shift{\mathbf{y}_2}{-\tau}-\shift{\mathbf{y}_1}{-\tau})^2 - l^2.
	\end{align*}
\end{subequations}
with constants $m_1, c, k, m_2, l,\mathsf{g}>0$.
To obtain a reformulation of \eqref{eq:PMSD:hybrid} which is suited for \Cref{alg:PantDDAE}, we rename $\lambda$ to $x_1$, $ \mathbf{y}_1$ to $x_3$, $\mathbf{y}_2$ to $x_4$ and rewrite \eqref{eq:PMSD:hybrid} as first-order \gls{DAE} given by 
\begin{subequations}
	\label{eq:PMSD:hybridFO}
	\begin{align*}
		\dot{x}_2 & = x_5, \\
		\dot{x}_3 & = x_6, \\
		\dot{x}_4 & = x_7, \\
		M\dot{x}_6 + C\dot{x}_3+ Kx_3 &= -2\shift{x_1}{-\tau} ( \shift{x_4}{-\tau}-\shift{x_3}{-\tau}) - m\mathsf{g},\\
		m \shift{\dot{x}_5}{-\tau} &= -2\shift{x_1}{-\tau} \shift{x_2}{-\tau},\\
		m \shift{\dot{x}_7}{-\tau} &= -2 \shift{x_1}{-\tau} (\shift{x_4}{-\tau}-\shift{x_3}{-\tau})- m\mathsf{g},\\
		0 &= \shift{x_2}{-\tau}^2 + (\shift{x_4}{-\tau}-\shift{x_3}{-\tau})^2 - l^2.
	\end{align*}
\end{subequations}
Since the equations $F_5$, $F_6$ and $F_7$ only depend on past time points we shift all three equations in the shifting step of \Cref{alg:PantDDAE} separately such that the differentiation during the shifting step yields no differentiation. No other equation is shifted. In the differentiation step we obtain the equations  $\dot{F}_1$, $\dot{F}_2$, $\dot{F}_3$, $F_4$, $\shift{F_5}{\tau}$, $\shift{F_6}{\tau}$ and $\shift{\ddot{F}_7}{\tau}$, which is the expected result \cite{Ung20}.

\section{Termination of the new algorithm}
\label{sec:termination}
Similarly as for the Pantelides algorithm we have to answer the question under which circumstances \Cref{alg:PantDDAE} applied to the \gls{DDAE} \eqref{eq:nlDDAE} terminates. As in the \gls{DAE} case (cf.\ \Cref{thm:TerminationPDAE}), we partition the set of variables $\setvars$ of \eqref{eq:nlDDAE} in such a way, that variables that only differ by the level of differentiation or shifting are grouped together. In more detail, we consider the equivalence relation
\begin{equation}
	\label{eq:equivRelDDAE}
	R_{\mathrm{equal}} \defEqual \left\{ (\elementsetvars, \tilde{\elementsetvars})\in\setvars\times\setvars \left|\,
	\begin{aligned}
		&\text{there exist } k,\ell \in \N,\,  p,q \in \N_0 \text{ and } \xi \in \set(x)\\ 
		&\text{such that } \elementsetvars = \shift{\xi^{(p)}}{k \tau},\,  \tilde{\elementsetvars} = \shift{\xi^{(q)}}{\ell \tau}
	\end{aligned}\right.  
	\right\}. 
\end{equation}
It is easy to see that if \Cref{alg:PantDDAE} applied to the \gls{DDAE} \eqref{eq:nlDDAE} terminates, then each equation is matched to a different variable and thus \eqref{eq:nlDDAE} is structurally nonsingular with respect to $\eqclass{\setvars}{R_{\mathrm{equal}}}$. The following result shows that this is indeed also a sufficient condition.

\begin{theorem}
\label{thm:TerminationPantDDAE3}
	Consider the \gls{DDAE} \eqref{eq:nlDDAE} with set of variables $\setvars$ and equivalence relation $R_{\mathrm{equal}}$ as defined in \eqref{eq:equivRelDDAE}. \Cref{alg:PantDDAE} applied to \eqref{eq:nlDDAE} terminates if and only if \eqref{eq:nlDDAE} is structurally nonsingular with respect to $\eqclass{\setvars}{R_{\mathrm{equal}}}$.
\end{theorem}

We split the proof of \Cref{thm:TerminationPantDDAE3} in two parts, namely the shifting and differentiation step, which are very similar to the original Pantelides algorithm. Thus we can extend the termination proof presented in \cite{Pan88} to our setting. The main idea in \cite{Pan88} is to show that differentiating a \gls{MSS} subset (see \Cref{def:structuralSingular}) renders it structurally nonsingular and a structurally nonsingular subset remains structurally nonsingular after differentiation. Let us first show that under the assumptions of \Cref{thm:TerminationPantDDAE3} the shifting step terminates.  

\begin{lemma}
	\label{lem:TerminationPantDDAE1}
	Consider the \gls{DDAE} \eqref{eq:nlDDAE} with set of variables $\setvars$ and equivalence relation $R_{\mathrm{equal}}$ as defined in \eqref{eq:equivRelDDAE}. The shifting step of \Cref{alg:PantDDAE} applied to \eqref{eq:nlDDAE} terminates if and only if \eqref{eq:nlDDAE} is structurally nonsingular with respect to $\eqclass{\setvars}{R_{\mathrm{equal}}}$.
\end{lemma}

\begin{proof}
	Recall that in the shifting step we group variables that are shifting similar, i.e., we work with the equivalence relation $R_{\mathrm{s}}$ from \Cref{def:shiftDiffGraph}. By virtue of \Cref{lem:AugmentpathMSS}, any subset of equations of the \gls{DDAE} \eqref{eq:nlDDAE} that \Cref{alg:PantDDAE}, or more precisely \Cref{alg:PantDAENew}, identifies to be shifted is \gls{MSS} with respect to $\eqclass{\setvars}{R_{\mathrm{s}}}$. Suppose that the subset
	\begin{align}
	\label{eq:generalMsSS}
		 \soi{F}(t,\tx, \dot{\ox}, \shift{\hx}{-\tau} ) = 0, 
	\end{align} 
	with $\set(\tx,\ox,\hx) \subseteq \set(x)$ of the \gls{DDAE} \eqref{eq:nlDDAE} is such a set. Note that here the sets 
	\begin{displaymath}
		\set(\tx),\ \set(\ox),\ \text{and}\ \set(\hx)
	\end{displaymath}
	are not necessarily disjoint. More precisely, decompose $\hx = (\hxa,\hxb)$ such that
	\begin{displaymath}
		\set(\hx)  = \set(\hxa) \dot{\cup} \set(\hxb),\qquad \set(\hxb) \cap \set(\ox, \tx )  = \emptyset,\qquad \text{and}\qquad \set(\hxa) \subseteq \set(\ox, \tx).
	\end{displaymath}
	Let us assume first that the \gls{DDAE} \eqref{eq:nlDDAE} is structurally nonsingular with respect to $\eqclass{\setvars}{R_\mathrm{equal}}$. This implies that $\set(\hxb)\neq\emptyset$. Shifting \eqref{eq:generalMsSS} results in the equation $\shift{\soi{F}}{\tau}$ given by
	\begin{displaymath}
		 \soi{F}(t+\tau,\shift{\tx}{\tau},\shift{\dot{\ox}}{\tau},\hx) = 0.
	\end{displaymath}
	By construction we have
	\begin{align*}
		\left|\set(\shift{\soi{F}}{\tau})\right| & = \left|\set( \soi{F}) \right|  \le \left|\eqclass{\set \begin{pmatrix} \tx, \dot{\ox}, \shift{\hx}{-\tau}
		\end{pmatrix}}{R_{\mathrm{equal} }}  \right|= \left|\eqclass{\set \begin{pmatrix} \tx, \dot{\ox}, \shift{\hx_2}{-\tau} \end{pmatrix}}{R_{\mathrm{equal} }} \right|\\
		& = \left|\eqclass{\set \begin{pmatrix} \tx, \dot{\ox}, \shift{\hx_2}{-\tau} \end{pmatrix}}{\Rs} \right| = \left|\eqclass{\set \begin{pmatrix} \shift{\tx}{\tau}, \shift{\dot{\ox}}{\tau},\hx_2 \end{pmatrix}}{\Rs} \right|
	\end{align*}
	showing that $\shift{\soi{F}}{\tau}$ is structurally nonsingular with respect to the equivalence relation $\Rs$. Note that this argument is easily extended to subsets of the general equation \eqref{eq:DDAEshiftDiffVars} and we conclude that whenever a \gls{MSS} subset of equations is shifted during the shifting step it becomes structurally nonsingular. In \Cref{alg:PantDAENew} we do not only replace $\soi{F}$ with $\shift{\soi{F}}{\tau}$ but may have to differentiate some of the equations. However, \Cref{prop:PDDAEdiffDoesntEffectShift} implies that this does not affect the shifting graph and hence does not effect the structural nonsingularity. In order to show that the shifting step terminates it suffices to show that any subset $\set (\overline{F}) \subseteq \set (F)$ that is disjoint from $\set{(\soi{F})}$ and is structurally nonsingular before we shift \eqref{eq:generalMsSS} remains structurally nonsingular after shifting. The only possibility for $\set (\overline{F})$ to become structurally singular is if a subset of $\set (\overline{F})$ has a matching edge to some vertex in $ \eqclass{\set(\tx, \dot{\ox})}{R_s}$. This implies that there is an alternating path from the exposed equation in $\soi{F}$ to one of the equations in $\set(\overline{F})$, a contradiction to $\set(\soi{F})\cap\set(\overline{F}) = \emptyset$. We conclude that the shifting step terminates.
	
	Conversely, assume that \eqref{eq:nlDDAE} is structurally singular with respect to $\eqclass{\setvars}{R_{\mathrm{equal}}}$. Then at some point \Cref{alg:PantDDAE} identifies a subset \eqref{eq:generalMsSS} of \eqref{eq:nlDDAE} that is structurally singular with respect to $\eqclass{\setvars}{R_\mathrm{equal}}$. In this case, we have $\set(\hxb) = \emptyset$ and thus
	\begin{align*}
		\left|\set(\soi{F})\right| > \left|\eqclass{\set \begin{pmatrix} \tx, \dot{\ox}, \shift{\hx}{-\tau} \end{pmatrix}}{R_{\mathrm{equal} }}  \right|= \left|\eqclass{\set \begin{pmatrix} \tx, \dot{\ox}, \shift{\hx_2}{-\tau} \end{pmatrix}}{\Rs} \right| \ge \left|\eqclass{\set \begin{pmatrix} \tx, \dot{\ox}\end{pmatrix}}{\Rs} \right|.
	\end{align*}		
	In this case, shifting does not increase the set of equivalence classes and thus the set \eqref{eq:generalMsSS} is still structurally singular with respect to the relation $\Rs$ and thus results in an infinite loop.
\end{proof}

To show that also the differentiation step in \Cref{alg:PantDDAE} terminates, we want to apply \Cref{thm:TerminationPDAE}. For this we have to show that the set of equations that we obtain after applying the shifting step and the trimmed linearization is structurally nonsingular with respect to
$\eqclass{\widetilde{\setvars}}{R_\mathrm{equal}}$ with 
\begin{align}
	\label{eq:setOfVarRestrictedToNonegativeShifts}
	\widetilde{\setvars} := \left\{ \elementsetvars \in \setvars \mid (\elementsetvars, \elementsetvars) \in \Rd \right\}.
\end{align}
Note that $\widetilde{\setvars}$ does not include variables depending on $t-\tau$ and hence $\eqclass{\widetilde{\setvars}}{R_\mathrm{equal}}$ is different from $\eqclass{{\setvars}}{R_\mathrm{equal}}$.

\begin{lemma}
\label{lem:TerminationPantDDAE2}
Suppose that \eqref{eq:nlDDAE} is structurally nonsingular with respect to $\eqclass{\setvars}{R_{\mathrm{equal}}}$. Let $G = (\VE \dot{\cup} \VV, E)$ with 
\begin{align*}
 	\VE = \left\{ \shift{\Cof_1^{(\ell_1)} }{\com_1 \tau }, \dots,  \shift{\Cof_\Com^{(\ell_\Com)} }{\com_\Com \tau }   \right\}.
\end{align*}
denote the graph that results from applying the shifting step and the trimmed linearization of \Cref{alg:PantDDAE} to \eqref{eq:nlDDAE}. Define the set of equations 
\begin{equation}
	\label{eq:DDAEafterShifting}
	0 = \overline{F}(\overline{t},\ox, \dot{\ox}, \shift{x}{-\tau}) = [\overline{F}_i]_{i=1,\dots,M}(\overline{t},\ox,\dot{\ox},\shift{x}{-\tau})
\end{equation}
 with $m\defEqual \max_i (\com_i)$, $\overline{\xi} = \begin{bmatrix} \shift{\xi}{0 \tau} & \dots & \shift{\xi}{m \tau}  \end{bmatrix}$ for $\xi \in \{t,x\}$ by 
$\overline{F}_i \defEqual \shift{F_i}{\com_i}$.
Then \eqref{eq:DDAEafterShifting} is structurally nonsingular with respect to $ \eqclass{\widetilde{\setvars}}{R_{\mathrm{equal} }}$ with $\widetilde{\setvars}$ defined in \eqref{eq:setOfVarRestrictedToNonegativeShifts}. 
\end{lemma}

\begin{proof}
	Since \eqref{eq:nlDDAE} is structurally nonsingular with respect to $\eqclass{\setvars}{R_{\mathrm{equal}}}$  \Cref{alg:PantDDAE} up to Line 6 terminates (cf.\ \Cref{lem:TerminationPantDDAE1}, \Cref{thm:trimmedLinearization}). Using \Cref{thm:trimmedLinearization}, we can assign each equation $\overline{F}_i $ for $i=1, \dots M$ to a highest shift. In other words, we can assign each equation to a variable which is shifted at least 0 times which yields directly the structural nonsingularity of $\oF$ with respect to  $\eqclass{\widetilde{\setvars}}{R_{\mathrm{equal} }}$.  
\end{proof}

Combining the previous results, we are now able to proof \Cref{thm:TerminationPantDDAE3}.

\begin{proof}[Proof of \Cref{thm:TerminationPantDDAE3}]
If \eqref{eq:nlDDAE} is structurally nonsingular with respect to $\eqclass{\setvars}{R_{\mathrm{equal}}}$  \Cref{alg:PantDDAE} up to Line 6 terminates (cf.\ \Cref{lem:TerminationPantDDAE1}, \Cref{thm:trimmedLinearization}). Since the resulting shifted equations \eqref{eq:DDAEafterShifting} are structurally nonsingular with respect to $\eqclass{\widetilde{\setvars}}{R_{\mathrm{equal} }}$, we can use \Cref{thm:TerminationPDAE} to conclude that \Cref{alg:PantDDAE} terminates. The converse direction follows from \Cref{lem:TerminationPantDDAE1}. 
\end{proof}

We conclude our analysis by investigating in which cases \Cref{alg:PantDDAE} determines that no equation needs to be shifted. The following result shows that this is the case if the \gls{DDAE} \eqref{eq:nlDDAE} is structurally nonsingular with respect to the set of equivalence classes that are obtained by using the relation $R_{\mathrm{equal}}$ with a restricted set of variables.

\begin{theorem}
	\label{thm:noShift}
	If the set of equations \eqref{eq:nlDDAE} is structurally nonsingular with respect to $\eqclass{\overline{\setvars}}{R_{\mathrm{equal}}}$ with
	\begin{align*}
		\overline{\setvars} \defEqual \left\{ \xi \in \set\left(x, \dot{x}\right) \,\left|\, \xi \text{ appears in } F \right.\right\} 
	\end{align*} 
	then \Cref{alg:PantDDAE} applied to \eqref{eq:nlDDAE} terminates and determines that no equation is shifted. 
	In this case also \Cref{alg:PantDAE} applied to \eqref{eq:nlDDAE} where we replace $\shift{x}{-\tau}$ with a function parameter $\lambda$ terminates and the resulting graphs of both algorithms are isomorphic.
\end{theorem}

\begin{proof}
	First observe that the structural nonsingularity of \eqref{eq:nlDDAE} with respect to $\eqclass{\overline{\setvars}}{R_\mathrm{equal}}$ implies structural nonsingularity with respect to $\eqclass{\setvars}{R_\mathrm{equal}}$. Thus \Cref{thm:TerminationPantDDAE3} ensures that \Cref{alg:PantDDAE} terminates if applied to \eqref{eq:nlDDAE}. Consider a subset
	\begin{displaymath}
		\widehat{F}(t, \ya, \yb, \dot{\yb}, \dot{\yc}, \shift{\hx}{-\tau}) = 0
	\end{displaymath}
	of the \gls{DDAE} \eqref{eq:nlDDAE} with $\set(\ya) \dot{\cup} \set(\yb) \dot{\cup} \set(\yc) \subseteq \set(x)$ and $\set(\hx) \subseteq \set(x) $. The structural nonsingularity with respect to $\eqclass{\overline{\setvars}}{R_{\mathrm{equal}}}$ implies
	\begin{align*}
		\left| \set(\widehat{F})  \right|  \le \left| \eqclass{\set\left(\ya, \yb, \dot{\yb}, \dot{\yc} \right)}{R_\mathrm{equal} }  \right|
   =  \left| \eqclass{\set\left( \ya, \yb, \dot{\yb}, \dot{\yc} \right)}{\Rs }  \right| 
\end{align*}
which is the structural nonsingularity of $\hF$ after deleting variables with lower shifts in the shifting graph. Thus no equation is shifted.
 We conclude that \Cref{alg:PantDDAE} can be reduced to the differentiation step. Within the differentiation step, \Cref{alg:PantDAE} and \Cref{alg:PantDAENew} coincide and thus \Cref{alg:PantDAE} terminates with the same result as \Cref{alg:PantDDAE}.
\end{proof}

We conclude our analysis by emphasizing that the Pantelides algorithm for \glspl{DDAE} suffers from the same problems as the original Pantelides algorithm (see for instance \Cref{ex:failurePantelidesDAE} and the discussion thereafter). In particular, there is no guarantee that the correct number of differentiations and shifts is identified.

\begin{example}
Even though the solution of the \gls{DDAE}
	\begin{align*}
		0  = x_4 + \shift{x_1}{-\tau} + f_1,\quad
		\dot{x}_2 & = x_1 + f_1, \quad
		0  = x_2 + x_4 + f_3, \quad
		0  = x_2 + \shift{x_3}{-\tau} + f_4
	\end{align*}
	depends on the shifted and differentiated equations $F_1, F_3, F_4$, the Pantelides algorithm for \glspl{DDAE} determines that we have to shift equations $F_1, F_3$, and $F_4$ but do not have to differentiate any equation. 
\end{example}
 
\section{Summary}
We have presented a method (\Cref{alg:PantDDAE}) to determine which equations of the \gls{DDAE} \eqref{eq:nlDDAE} need to be shifted and which equations need to be differentiated. The algorithm extends the Pantelides algorithm \cite{Pan88} for \glspl{DAE} to the \gls{DDAE} case. The main idea that enables the generalization is the introduction of equivalence classes in the bipartite graph associated with the \gls{DDAE} \eqref{eq:nlDDAE}. For further details see \Cref{def:graphDAE}. The Pantelides algorithm for \glspl{DDAE} is divided into the shifting step and the differentiation step. We prove (\Cref{prop:PDDAEdiffDoesntEffectShift}) that differentiation does not affect the shifting graph and hence start with the shifting step. It turns out that already in the shifting step we have to differentiate some equations. To obtain the required number of differentiations for each equation, we may have to solve several linear systems, see \cref{subsec:DiffInShift} for further details. We presented a necessary and sufficient condition for the termination of \Cref{alg:PantDDAE} in \Cref{thm:TerminationPantDDAE3}. We foresee that we can combine our framework with a dummy derivative approach \cite{MatS93} and an algebraic approach \cite{SchS16a} such that \Cref{alg:PantDDAE} can be used with numerical time integration methods.

Since our framework builds upon the Pantelides algorithm, it has similar strengths and weaknesses. One of the big advantages of our algorithm (in contrast to the algebraic index reduction procedure in \cite{HaM16}) is that the underlying graph-theoretical tools do not suffer from numerical rank decisions and can be computed efficiently also in a large-scale setting.  On the other hand, our algorithm is not invariant under equivalence transformations and thus may fail to determine the correct number of shifts and differentiations. In the \gls{DAE} case, this fact is accounted for in the $\Sigma$-method \cite{Pry01, PryNeTa15} by including a success check. It is ongoing research to develop such a success check for \glspl{DDAE}.

\bibliographystyle{plain}
\bibliography{ms}    

\vfill
\begin{minipage}{\linewidth}
	{\printglossary[type=acronym,title=List of Abbreviations,toctitle=List of Abbreviations]}
\end{minipage}

\end{document}